\documentclass[reqno]{amsart}
\usepackage[foot]{amsaddr}

\input macros.tex

\begin{document}

\makeatletter
\newcommand{\definetitlefootnote}[1]{
  \newcommand\addtitlefootnote{
    \makebox[0pt][l]{$^{*}$}
    \footnote{\protect\@titlefootnotetext}
  }
  \newcommand\@titlefootnotetext{\spaceskip=\z@skip $^{*}$#1}
}
\makeatother
\definetitlefootnote{
  This is an extended version of the reference~\cite{ramics2018}}

\title{The continuous weak order\addtitlefootnote} 

\author[M. J. Gouveia]{Maria Jo\~{a}o Gouveia$^1$}
\address{
  $^1$Faculdade de Ci\^{e}ncias, Universidade de
  Lisboa, Portugal
}
\email{mjgouveia@fc.ul.pt}
\thanks{$^1$Partially supported by FCT under grant
SFRH/BSAB/128039/2016}

\author[L. Santocanale]{Luigi Santocanale$^2$}
\address{
  $^2$LIS, CNRS UMR 7020, Aix-Marseille Universit\'e,
  France
}
\email{luigi.santocanale@lis-lab.fr}

\maketitle

\begin{abstract}
  
The set of permutations on a finite set can be given the lattice
structure known as the weak Bruhat order. This lattice structure is
generalized to the set of words on a fixed alphabet
$\Sigma = \{\,x,y,z,\ldots \,\}$, where each letter has a fixed number
of occurrences. These lattices are known as multinomial lattices and,
when $\mathrm{card}(\Sigma) = 2$, as lattices of lattice paths.  By
interpreting the letters $x,y,z,\ldots $ as axes, these words can be
interpreted as discrete increasing paths on a grid of a
$d$-dimensional cube, with $d = \mathrm{card}(\Sigma)$.

We show how to extend this ordering to images of continuous monotone
functions from the unit interval to a $d$-dimensional cube and prove
that this ordering is a lattice, denoted by $\LId$.  This construction
relies on a few algebraic properties of the quantale of
join-continuous functions from the unit interval of the reals to
itself: it is cyclic $\star$-autonomous and it satisfies the mix rule.

We investigate structural properties of these lattices, which are
self-dual and not distributive.  We characterize \jirr elements and
show that these lattices are generated under infinite joins from their
\jirr elements, they have no \cjirr elements nor compact elements.  We
study then embeddings of the $d$-dimensional multinomial lattices into
$\LId$. We show that these embeddings arise functorially from
subdivisions of the unit interval and observe that $\LId$ is the
Dedekind-MacNeille completion of the colimit of these embeddings.
Yet, if we restrict to embeddings that take rational values and if
$d > 2$, then every element of $\LId$ is only a join of meets of
elements from the colimit of these embeddings.

\end{abstract}


\section{Introduction}

The weak Bruhat order \cite{GuRo63,YaOk69} on the set of permutations
of an $n$-element set, also known as \Permutohedron, see \cite{STA0}
for an elementary exposition, is a lattice structure which has been
widely studied in view if its close connections to combinatorics and
geometry, see e.g. \cite{Bjo84,bjornerbrenti,STA2-9,STA2-10}.  Its
algebraic structure has also been investigated and, by now, is well
understood \cite{Casp00,SaWe13,JEMS}.

Multinomial lattices \cite{BB,Flath93,RAMIREZALFONSIN2002,ORDER-24-3},
or lattices of multipermutations, generalize \Permutohedra in a
natural way.  Elements of a multinomial lattice are multipermutations,
namely 
words on a totally ordered finite alphabet
$\Sigma = \{\,x,y,z\ldots \,\}$ with a fixed number of occurrences of
each letter. The \wo on multipermutations is the reflexive and
transitive closure of the binary relation $\covered$ defined by
$wabu \covered wbau$, for $a,b \in \Sigma$ and $a < b$.  If each
letter of the alphabet has exactly one occurrence, then these words
are permutations and the ordering is the weak Bruhat ordering.
Multinomial lattices embed into \Permutohedra as principal ideals;
possibly, this is a reason for the lattice theoretic literature on
them not to be contained. Multipermutations have, however, a strong
geometrical flavour that in our opinion justifies exploring further
their lattice theoretic structure.
These words can be given a geometrical interpretation as discrete
increasing paths in some Euclidean cube of dimension
$d = \card(\Sigma)$; the \wo can be thought of as a way of organizing
these paths into a lattice structure. When $\card(\Sigma) = 2$, the
connection with geometry is well-established: in this case these
lattices are also known as lattices of lattice paths with North and
East steps \cite{pinzani}; the objects these lattices are made of are
among the most studied in enumerative combinatorics
\cite{Krattenthaler97,BANDERIER2002} and many counting results are
implicitly related to the order and lattice structures.  We did not
hesitate in \cite{ORDER-24-3} to call the multinomial lattices
``lattices of paths in higher dimensions''.  Willing to understand the
geometry of higher dimensional multinomial lattices, we started
wondering whether there are full geometric relatives of these
lattices. More precisely, we asked whether the \wo can be extended
from discrete paths to continuous increasing paths. We present in this
paper our answer to this question.  Our main result sounds as follows:
\begin{theorem*}
  Let $d \geq 2$.  Images of increasing continuous paths from
  $\vec{0}$ to $\vec{1}$ in $\mathbb{R}^{d}$ can be given the
  structure of a lattice; moreover, all the \Permutohedra and all the
  multinomial lattices can be embedded into one of these lattices
  while respecting the dimension $d$.
\end{theorem*}
We call this lattice the \emph{\cwo} in dimension $d$.
While a proof of the above statement was 
available a few years ago, only recently we could structure and ground
that proof on a solid algebraic setting, making it possible to further
study these lattices.
The algebra we consider is the one of the quantale $\LjI$ of \jc
functions from the unit interval of the reals to itself. This is a
\saq, see \cite{barr79}, and moreover it satisfies the mix rule, see
\cite{cockettSeely}.
The construction of the \cwo is actually an instance of a general
construction of a lattice $\Ld{Q}$ from a \saq $Q$ satisfying the mix
rule. When $Q = \two$ (the two-element Boolean algebra) this
construction yields the usual \wBo on permutations; when $Q = \LjI$,
this construction yields the \cwo. Moreover, when $Q$ is the quantale
of \jc functions from the finite chain $\set{0,1,\ldots ,n}$ to
itself, this construction yields a multinomial lattice. The functorial
properties of this construction are a key tool for analysing various
embeddings.
The step we took can be understood as an instance of moving to a
different set of (non-commutative, in this case) truth values, as
notably suggested in \cite{lawvere}.

\medskip 
Let us state our algebraic results.  Let
$\langle Q, 1,\otimes,\opp{}\rangle$ be a cyclic non-commutative \saq
satisfying the MIX rule. That is, we require that
$x \otimes y \leq x \oplus y$, for each $x,y \in Q$, where $\oplus$
is the monoid structure dual to $\otimes$.
Let $d \geq 2$, $\cd := \{\,(i,j) \mid 1\leq i < j \leq d\,\}$ and
consider the product $Q^{\cd}$. Say that a tuple $f \in Q^{\cd}$ is
\emph{closed} if $f_{i,j} \otimes f_{j,k} \leq f_{i,k}$ (each
$i < j < k$), and that it is \emph{open} if
$f_{i,k} \leq f_{i,j} \oplus f_{j,k}$ (each $i < j < k$). Say that $f$
is \emph{clopen} if it is closed and open. Under these conditions, the
following statement hold:
\begin{theorem*}
  The set of clopen tuples of $Q^{\cd}$ is, with the pointwise ordering,
  a lattice, noted $\Ld{Q}$. The construction $\Ld{-}$ yields a
  limit preserving functor to the category of lattices.
\end{theorem*}
We shall make later in the text precise the domain of this functor.
Paired with the following statement, relating the algebraic structure
of $\QjI$ to the reals, we obtain a proof the main result stated
above.
\begin{theorem*}
  Clopen tuples of $\LjI^{\cd}$ bijectively correspond to images of
  monotonically increasing continuous functions $\p : \I \rto \I^{d}$
  such that $\p(0) = \vec{0}$ and $\p(1) = \vec{1}$.
\end{theorem*}
Let us mention that motivations for developing this work also
originated from various researches undergoing in theoretical computer
science, modelling the behaviour of concurrent processes via directed
homotopy \cite{goubault2003,Grandis2005} and discrete approximation of
continuous paths via words \cite{reutenauer}.  The relationship
between directed homotopies and congruences of two-dimensional
multinomial lattices was discussed in \cite{ORDER-24-3}. The
connection with discrete geometry appears in the conference version of
this work \cite{ramics2018}. In both cases it was distinct to us the
need of developing the mathematics of a continuous \wo in dimension
$d \geq 3$.

\bigskip

The paper is organized as follows.  We recall in \secRef{notation}
some definitions and elementary results, mainly on \jcont (and \mcont)
functions and adjoints.
In \secRef{bisemigroups} we identify the least algebraic structure
needed to perform the construction of the lattice $\Ld{Q}$.
Therefore, we introduce and study mix \lbs{s} which, in the cases of
interest to us, arise from \msaq{s}.
\secRef{quantalesFromChains} proves that if $I$ is what we call a
perfect chain, then the quantale of \jc functions from $I$ to itself
is mix \staraut. 
Finite chains and the unit interval of the real numbers are examples
of perfect chains.
\secRef{latticesFromQuantales} describes the construction of the
lattice $\Ld{Q}$, for an integer $d \geq 2$ and a \lbs $Q$.
In \secRef{quantaleLI} we focus on the particular structure of $\LjI$,
the quantale of continuous functions from the unit interval to itself.
\secRef{paths} defines the central notion of path and discusses its
equivalent characterizations.  In \secRef{pathsDimensionTwo} we show
that paths in dimension $2$ are in bijection with elements of the
quantale $\QjI$.  In \secRef{pathsDimensionMore} we argue that paths
in higher dimensions bijectively correspond to clopen tuples of the
product lattice $\QjI^{\cd}$, that is, to elements of $\Ld{\QjI}$.
In \secRef{weakBruhat} we discuss some structural properties of the
lattices $\LjI$; in particular we characterize \jirr elements of these
lattices and argue that these lattices do not have any \cjirr element
nor any compact element.
In \secRef{embeddings} we argue that embeddings from multinomial
lattices into the \cwo functorially arise from complete maps of
perfect chains.
Finally, in \secRef{generation}, we argue that if we restrict to the
embeddings of multinomial lattices obtained from splitting the unit
interval into $n$ intervals of the same size, then the \cwo is not the
\DMNc of the colimit of these embeddings, yet every element is a join
of meets (and a meet of joins) of elements from such a colimit.

\section{Elementary facts on \jcont functions}
\label{sec:notation}

Throughout this paper, $\setof{d}$ shall denote the set
$\set{1,\ldots ,d}$ while we let
$\cd := \set{(i,j) \mid 1 \leq i < j \leq d}$.

Let $P$ and $Q$ be complete posets; a function $f : P \rTo Q$ is
\emph{\jcont} (\resp \emph{\mcont}) if
\begin{align}
  \label{eq:meetcontinuous}
  f(\bigvee X) & = \bigvee_{x \in X} f(x)\,, 
  \;\;\;\;(\text{\resp}\; f(\bigwedge X) = \bigwedge_{x \in X} f(x))\,,  
\end{align}
for every $X\subseteq P$ such that $\bigvee X$ (\resp $\bigwedge X$)
exists. We say that $f$ is \emph{\bcont} if it is both \jc and \mc.

Recall that $\bot_{P} := \bigvee \emptyset$ (\resp
$\top_{P} := \bigwedge \emptyset$) is the least (\resp greatest)
element of $P$. Note that if $f$ is join-continuous (\resp \mcont)
then $f$ is monotone and $f(\bot_P)=\bot_Q$ (\resp
$f(\top_P)=\top_Q$).
Let $f$ be as above; a map $g : Q \rto P$ is \emph{\LADJ} to $f$ if
$g(q) \leq p$ holds if and only if $q \leq f(p)$ holds, for each
$p \in P$ and $q \in Q$; it is \emph{\RADJ} to $f$ if $f(p) \leq q$ is
equivalent to $p \leq g(q)$, for each $p \in P$ and $q \in Q$.  Notice
that there is at most one function $g$ that is \LADJ (\resp \RADJ) to
$f$; we write this relation by $g = \ladj{f}$ (\resp $g =
\radj{f}$). Clearly, when $f$ has a \RADJ, then
$f = \ladj{(\radj{g})}$, and a similar formula holds when $f$ has a
\LADJ.  We shall often use the following fact:
\begin{lemma}
  If $f : P \rto Q$ is monotone and $P$ and $Q$ are two complete
  posets, then the following are equivalent:
  \begin{enumerate}
  \item $f$ is \jcont (\resp \mcont),
  \item $f$ has a \RADJ (\resp \LADJ).
  \end{enumerate}
\end{lemma}
If $f$ is \jcont (\resp \mcont), then we have
\begin{align*}
  \radj{f}(q) & = \bigvee \set{p \in P \mid f(p) \leq q} \qquad
  (\,\text{\resp} \;\;\ladj{f}(q) = \bigwedge \set{p \in P \mid q \leq
    f(p)}\,)\,, \tag*{for each $q \in Q$.}
\end{align*}
Moreover, if $f$ is surjective, then these formulas can be
strengthened so to substitute inclusions with equalities:
\begin{align}
  \label{eq:adjsurjective}
  \radj{f}(q) & = \bigvee \set{p \in P \mid  f(p) = q}
  \qquad (\,\text{\resp}
  \;\;\ladj{f}(q) = \bigwedge \set{p \in P \mid  q = f(p)}\,)\,, \\
  &\mbox{\hspace{80mm}}\tag*{for each $q \in Q$.}
\end{align}
The set of monotone functions from $P$ to $Q$ can be ordered
point-wise: $f \leq g$ if $f(p) \leq g(p)$, for each $p \in
P$. Suppose now that $f$ and $g$ both have \RADJ{s}; let us argue that
$f \leq g$ implies $\radj{g} \leq \radj{f}$: for each $q \in Q$, the
relation $\radj{g}(q) \leq \radj{f}(q)$ is obtained by transposing
$f(\radj{g}(q)) \leq g(\radj{g}(q)) \leq q$, where the inclusion
$g(\radj{g}(q)) \leq q$ is the counit of the adjunction.  Similarly,
if $f$ and $g$ both have \LADJ{s}, then $f \leq g$ implies
$\ladj{g} \leq \ladj{f}$.

Let $P$ be a poset, and let $\iota : P \rto Q$ be an embedding of $P$
into a complete lattice $Q$. Such embedding is a \emph{\DMNc} if
$\iota$ is \bcont and, for each $q \in Q$, there are sets
$X,Y \subseteq P$ such that
$q =\bigvee_{x \in X} \iota(x) = \bigwedge_{y \in Y} \iota(y)$. The
\DMNc is unique up to isomorphism.

\section{\LBS{s}}
\label{sec:bisemigroups}

A (\nc, \bd) \emph{\Lbs} (\lbs, for short) is a structure
$\langle Q, \bot,\vee,\top,\land,\otimes,\oplus\rangle$ where
$\langle Q, \bot,\vee,\top,\land\rangle$ is a bounded lattice,
$\otimes$ is a binary associative operation on $Q$ which distributes
overs finite joins, $\oplus$ is a binary associative operation on $Q$
which distributes over finite meets; moreover, the following relations
\begin{align}
    \label{eq:hemidistr1p}
    \beta \otimes (\gamma \oplus \delta) & \leq (\beta \otimes \gamma)
    \oplus \delta \,,\\
    \label{eq:hemidistr2p}
    (\alpha \oplus \beta) \otimes \gamma & \leq
    \alpha \oplus (\beta \otimes \gamma)\,.
\end{align}
holds, for each $\alpha,\beta,\gamma,\delta \in Q$. We call these
inclusions \emph{hemidistributive} laws.  We say that an \lbs is
\emph{mix} if the relation
\begin{align}
  \label{eq:mix}
  \alpha \otimes \beta & \leq \alpha \oplus \beta\,.
\end{align}
holds, for each $\alpha,\beta \in Q$.  We call this inclusion the
\emph{mix} rule.
The inclusions \eqref{eq:hemidistr1p} and \eqref{eq:hemidistr2p} are
non-commutative versions of the hemidistributive law of \cite[\S
6.9]{dunn2001algebraic} and are related to the weak distributivity of
\cite{cockettSeely2}. The mix rule \eqref{eq:mix} is well known in
proof theory, see e.g. \cite{cockettSeely}.

\begin{remark}
  All the \lbs{s} that we shall consider have units; therefore, they
  are (possibly non-commutative) \lbm{s} in the sense of
  \cite{GalatosLADT2018}.  We use $1$ (\resp $0$) to denote the unit
  of the operation $\otimes$ (\resp of $\oplus$) of an \lbm.  The
  signature of \lbm{s} is obtained by adding the two unit constants to
  the signature of \lbs{s}.
  Let us emphasize, however, that the morphisms between \lbm{s} that
  we shall consider do not, in general, preserve units.  This is the
  reason for which we emphasize the weaker structure of \lbs.
\end{remark}

We shall also use the following generalized
hemidistributive laws:
\begin{align}
  \label{eq:hemidistr1}
  (\alpha \oplus \beta) \otimes (\gamma \oplus \delta) &\leq
  \alpha \oplus (\beta \otimes \gamma) \oplus \delta\,, \\
  \alpha \otimes (\beta \oplus \gamma) \otimes \delta&\leq
  (\alpha \otimes \beta) \oplus (\gamma \otimes \delta)\,,
  \label{eq:hemidistr2} 
\end{align}

\begin{lemma}
  The inclusions \eqref{eq:hemidistr1} and \eqref{eq:hemidistr2} are
  derivable from \eqref{eq:hemidistr1p} and \eqref{eq:hemidistr2p}.
  Moreover, in the extended language  of \lbm{s} (using units)
  these pairs of inclusions are equivalent and the mix rule
  \eqref{eq:mix} is equivalent to $0 \leq 1$.
\end{lemma}
\begin{proof}
  Having both \eqref{eq:hemidistr1p} and \eqref{eq:hemidistr2p}, 
  we derive \eqref{eq:hemidistr1} as follows:
  \begin{align*}
    (\alpha \oplus \beta) \otimes (\gamma \oplus \delta) & \leq \alpha
    \oplus (\beta \otimes (\gamma \oplus \delta)) \leq \alpha \oplus
    (\beta \otimes \gamma) \oplus \delta\,.
  \end{align*}
  Using units, we obtain \eqref{eq:hemidistr1p} from \eqref{eq:hemidistr1} by
  instantiating $\alpha$ to $0$; we obtain \eqref{eq:hemidistr2p} from
  \eqref{eq:hemidistr1} by instantiating $\delta$ to $0$. 
  For the last statement, if \eqref{eq:mix} holds, then $0 \leq 1$
  is derived by instantiating in \eqref{eq:mix} $\alpha$ with $0$ and
  $\beta$ with $1$.
  Conversely, suppose that $0 \leq 1$ and observe then that
  $0 \otimes 0 \leq 0 \otimes 1 = 0$.  Letting $\beta = \gamma = 0$
  in \eqref{eq:hemidistr1}, we derive \eqref{eq:mix} as follows:
  \begin{align*}
    \alpha \otimes \delta & = (\alpha \oplus 0) \otimes (0 \oplus
    \delta) \leq \alpha \oplus (0 \otimes 0) \oplus \delta \leq
    \alpha \oplus 0 \oplus \delta = \alpha \oplus \delta\,.
    \tag*{\qedhere}
  \end{align*}
\end{proof}

\medskip

All the \lbs{s} that we shall consider arise from \nc \bd \irl.

\smallskip

A (\nc, \bd) \emph{\rl} is a structure
$\langle Q,\bot,\vee,\top,\land,1,\otimes,\lrimpl,\rlimpl\rangle$ such
that $\langle Q,\bot,\vee,\top,\land\rangle$ is a \bd lattice,
$\langle Q, 1,\otimes\rangle$ is a monoid structure compatible with
the lattice ordering (noted $\leq$) which moreover is related to the
binary operations $\lrimpl,\rlimpl$ as follows: 
\begin{align}
  \label{eq:residuation}
  \alpha \otimes \beta & \leq \gamma \quad \tiff\quad \alpha \leq
  \gamma\rlimpl \beta \quad \tiff\quad \beta \leq \alpha \lrimpl
  \gamma\,, \quad \text{for each $\alpha,\beta,\gamma \in Q$}.
\end{align}
The operations $\lrimpl,\rlimpl$ are called the residuals (or
adjoints) of $\otimes$. Let us recall that the following inclusions
are valid:
\begin{align}
  \label{eq:units}
  \alpha \otimes (\alpha \lrimpl \beta) & \leq \beta\,,
  & (\beta \rlimpl \alpha)\otimes \alpha & \leq \beta\,.
\end{align}

A (\emph{unital}) \emph{quantale} \cite{rosenthal1990} is a complete
lattice $Q$ coming with a monoid structure $1,\otimes$ such that
$\otimes$ distributes over arbitrary joins in both variables. A
quantale is a \rl in a canonical way, as distribution over arbitrary
joins ensures the existence of the residuals.

\smallskip

A \rl is said to be
\emph{involutive} if it comes with an element $0 \in Q$ such that
\begin{align*}
  x \lrimpl 0 & = 0 \rlimpl x\,,
  \\
   0\rlimpl  (x \lrimpl 0) & = x\,,
\end{align*}
for each $x \in Q$. Such an element $0$ is called \emph{cyclic} and
\emph{dualizing}.
  In \cite{ramics2018} we called a complete \irl a \emph{\saq}, as these
  structures are posetal version of \staraut categories
  \cite{barr79}. Similar namings, such as \emph{(pseudo) \staraut
    lattice}, have also been used in the literature, see
  e.g. \cite{Paoli2005,Emanovsky2008}. We shall stick to this naming
  in the future sections as all the \irl{s} that we consider are
  complete.
Given an \irl
$\langle Q,\bot,\vee,\top,\land,1,\otimes,\lrimpl,\rlimpl,0\rangle$,
we obtain an \lbm by defining
\begin{align}
  \label{def:oppoplus}
  \opp{x} & := x \lrimpl 0\,,
  &
  f \oplus g &  := \Opp{(\opp{g} \otimes
    \opp{f})}\,.
\end{align}
From these definitions it follows that $\oppfun$ is an antitone
involution of $Q$ and that $0 = \opp{1}$. Moreover, considering that
\begin{align*}
  \opp{(x \otimes y)} & = y \lrimpl \opp{x} = \opp{y} \rlimpl x\,, \\
  \quad \opp{x} \oplus y  & = \Opp{(\opp{y} \otimes x)} = x \lrimpl
   \OPP{\opp{y}} 
    = x \lrimpl y\,,
   \\ \quad
   x \oplus \opp{y} & = x \rlimpl y\,,
 \end{align*}
 the relations in \eqref{eq:residuation} can be expressed as follows:
\begin{align}
  \alpha \otimes \beta & \leq \gamma \quad \tiff\quad \alpha \leq
  \gamma\oplus \opp{\beta}\quad \tiff\quad \beta \leq
  \opp{\alpha}\oplus \gamma\,, \quad \text{for each
    $\alpha,\beta,\gamma \in Q$}.
  \label{eq:residuationoplus}
\end{align}
\begin{lemma}
  With the definitions given in equation~\eqref{def:oppoplus}, each
  \irl is a \lbm, and therefore an \lbs.
\end{lemma}
\begin{proof}
  Since $0,\oplus$ are dual to $1,\otimes$, $\oplus$ is a monoid
  operation on $Q$ with unit $0$ and which distributes over meets.

  We therefore verify that the hemidistributive laws holds in $Q$.
  Considering that
  $\alpha \oplus \beta = \oppopp{\alpha} \oplus \beta = \opp{\alpha}
  \lrimpl \beta$ and, similarly,
  $\gamma \oplus \delta = \gamma \rlimpl \opp{\delta}$, we derive
  \begin{align*}
    \opp{\alpha} \otimes (\alpha \oplus \beta) \otimes (\gamma
    \oplus \delta) \otimes \opp{\delta} & = \opp{\alpha} \otimes
    (\opp{\alpha} \lrimpl \beta) \otimes (\gamma \rlimpl
    \opp{\delta}) \otimes \opp{\delta} \leq \beta \otimes \gamma \,,
  \end{align*}
  using \eqref{eq:units}. Yet, the inequality so deduced is equivalent
  to \eqref{eq:hemidistr1} by adjointness \eqref{eq:residuationoplus}.
\end{proof}

According to our previous observations, we could have defined a \irl
as a structure $\langle Q, \lattsig, 1, \otimes, 0, \oppfun\rangle$
where $\langle Q,\lattsig \rangle$ is a \bd lattice, $\otimes$ is a
monoid operation (with unit $1$) on $Q$ that distributes over joins,
$ \oppfun : Q \rto Q$ is an antitone involution of $Q$, subject to the
residuation laws as in \eqref{eq:residuationoplus}, where the
structure $(0,\oplus)$ on $Q$ is defined by duality:
\begin{align}
  0 & := \opp{1} \quad \tand \quad f \oplus g := \Opp{(\opp{g} \otimes
    \opp{f})}\,.
  \label{eq:defopluszero}
\end{align}
This shall be our preferred way to verify that a \rl with a distinct
element $0$ is an \irl. For the sake of verifying that a structure is
an \irl, let us remark that we can simplify our work according to the
following statement.
\begin{lemma}
  \label{lemma:verifHalf}
  Consider a structure $\langle Q, \lattsig, 1, \otimes, 0,
  \oppfun\rangle$ as above, where we only require that
  $\alpha \otimes \beta \leq \gamma$ is equivalent to $\alpha \leq
  \gamma\oplus \opp{\beta}$, for each $\alpha,\beta,\gamma \in Q$.
  Then $\alpha \otimes \beta \leq \gamma$ is also equivalent to
  $\beta \leq
  \opp{\alpha}\oplus \gamma$, for each $\alpha,\beta,\gamma \in Q$.
\end{lemma}
\begin{proof}
  Suppose that $\alpha \otimes \beta \leq \gamma$, so
  $\alpha \leq \gamma\oplus \opp{\beta}$. Apply $\oppfun$ to
  this relation and derive
  $\beta \otimes \opp{\gamma} = \OPP{(\gamma\oplus \opp{\beta})} \leq
  \opp{\alpha}$; derive then
  $\beta \leq \opp{\alpha} \oplus \OPP{\opp{\gamma}} = \opp{\alpha}
  \oplus \gamma$. For the converse direction, observe that all these
  transformations are reversible.
\end{proof}

\begin{example}
  Boolean algebras are the \irl{s} such that $\wedge = \otimes$ and
  $\vee = \oplus$. 
  Similarly, distributive lattices are the \lbs{s} such that
  $\wedge = \otimes$ and $\vee =
  \oplus$. 
\end{example}

\begin{example}
  \label{ex:Sugihara}
  Consider the following structure on the
  ordered set $\set{-1 < 0 <1}$:
  \begin{align*}
    \begin{array}{r@{\;\;}|@{\;\;}rrr}
      \otimes&-1&0&1\\\hline
      -1 &-1&-1&-1\\
      0 &-1&0&1\\
      1 &-1&1&1
    \end{array}
    \qquad
    \begin{array}{r@{\;\;}|@{\;\;}rr@{\;\;\,}r}
      \oplus&-1&0&1\\\hline
      -1 &-1&-1&1\\
      0 &-1&0&1\\
      1 &1&1&1
    \end{array}
    \qquad
    \begin{array}{r@{\;\;}|@{\;\;}r}
       &\star\\\hline
      -1 &1\\
      0 &0\\
      1 &-1
    \end{array}
  \end{align*}
  Together with the lattice structure on the chain, this structure
  yields a mix \irl, known in the literature as the Sugihara monoid on
  the three-element chain, see e.g. \cite{GALATOS20122177}.
\end{example}

\begin{example}
  \label{ex:jcfunctions}
  As the category of complete lattices and \jcont functions is a
  symmetric monoidal closed category, for every complete lattice $X$
  the set of \jcont functions from $X$ to itself is a monoid object in
  that category, that is, a quantale, see
  \cite{joyaltierney,rosenthal1990}, and therefore a \rl.
  We review this next.
  For  a complete lattice $X$, let $\Qj(X)$ denote the set of \jc functions
  from $X$ to itself. For $f,g \in \Qj(X)$ define
  $f \otimes g := g \circ f$. Considering that the ordering in
  $\Qj(X)$ is pointwise, let us verify that $\otimes $ distributes
  over arbitrary joins:
  \begin{align*}
    ((\bigvee_{i \in I} f_{i})  \otimes g)(x) & =
    (g \circ \bigvee_{i \in I} f_{i})(x)
    = g((\bigvee_{i \in I} f_{i})(x))
    = g(\bigvee_{i \in I} f_{i}(x))
    \\
    &
    = \bigvee_{i \in I} g(f_{i}(x))
    = \bigvee_{i \in I} ((g \circ f_{i})(x))
    = (\bigvee_{i \in I} g \circ f_{i})(x) = (\bigvee_{i \in I} (f_{i}  \otimes g))(x)\,, \\
    (f \otimes (\bigvee_{i \in I} g_{i}))(x)
    & =
    ((\bigvee_{i \in I} g_{i}) \circ f)(x)  = (\bigvee_{i \in I}
    g_{i})(f(x))
    \\
    &
    = \bigvee_{i \in I} g_{i}(f(x)) 
    = (\bigvee_{i \in I} (g_{i} \circ f))(x)
    = (\bigvee_{i \in I} (f \otimes g_{i}))(x)
    \,.
  \end{align*}
  Obviously, the identity is the unit for $\otimes$.
  We argue in the next Section that if $I$ is a finite chain or the
  interval $[0,1]$, then $\Qj(I)$ has a cyclic dualizing element, thus
  a \irl extending the \rl
  structure. 
\end{example}

\section{Mix \saq{s} from perfect chains}
\label{sec:quantalesFromChains}

We consider complete chains $I$ such that the two transformations
\begin{align}
  \label{eq:meetofjoinof}
  \meetof{f}(x) & = \bigwedge_{x < x'} f(x')\,,
  & 
  \joinof{f}(x) & = \bigvee_{x' < x} f(x')\,.
\end{align}
yield an order isomorphism from $\Lj(I)$ to $\Lm(I)$. We shall say
that such a chain is \emph{\perfect}.
\begin{example}
  Let $n \geq 0$ and let $\In$ be the chain $\set{0,\ldots ,n}$. A \jc
  function from $\In$ to $\In$ is uniquely determined by the value on
  the set $\set{1,\ldots ,n}$ of its \jp elements. Similarly, a \mc
  function from $\In$ to $\In$ is uniquely determined by its
  restriction to the set $\set{0,\ldots ,n-1}$ of its \mp elements.
  We immediately deduce that $\Lj(\In)$ and $\Lm(\In)$ are order
  isomorphic. The functions defined in~\eqref{eq:meetofjoinof} realize
  this isomorphism.
  Observe that, for $I = \In$, we have
  \begin{align*}
    \meetof{f}(x) & =
    \begin{cases}
      n\,, & x = n\,, \\
      f(x +1)\,, & \ttoth\,,
    \end{cases}
    &
    \joinof{f}(x) & 
    =
    \begin{cases}
      0\,, & x = 0\,, \\
      f(x -1)\,, & \ttoth\,.
    \end{cases} 
  \end{align*}
\end{example}
\begin{example}
  We shall see with Proposition~\ref{prop:meet-cont-closure} that the
  interval $[0,1]$ of the reals, later on denoted by $\I$, is
  perfect. The quantale $\Qj(\I)$ shall be investigated further in
  Section~\ref{sec:quantaleLI}.
\end{example}

Recalling that the correspondences sending $f \in \Lj(I)$ to
$\radj{f} \in \Lm(I)$ and $g \in \Lm(I)$ to $\ladj{g} \in \Lj(I)$ are
inverse is antitone, let us observe the following:
\begin{proposition}
  \label{lemma:meetofjoinofTwo}
  For each $f \in \Lj(I)$, the
  relation $\joinof{(\radj{f})} = \ladj{(\meetof{f})}$ holds.
  Therefore, the function $\oppfun$ defined by
  \begin{align*}
    \opp{f} & := \joinof{(\radj{f})} = \ladj{(\meetof{f})}\,,
  \end{align*}
   is an involution of $\Lj(I)$.
\end{proposition}
\begin{proof}
  Let $f \in \Qj(I)$; we shall argue that $\joinof{(\radj{f})}$ is left
  adjoint to $\meetof{f}$, namely that $x \leq \meetof{f}(y)$ if and
  only if $\joinof{(\radj{f})}(x) \leq y$, for each $x,y \in \I$.
  
  We begin by proving that $x \leq \meetof{f}(y)$ implies
  $\joinof{(\radj{f})}(x) \leq y$.  Suppose $x \leq \meetof{f}(y)$ so,
  for each $z$ with $y < z$, we have $x \leq f(z)$. Suppose that
  $\joinof{(\radj{f})}(x) \not\leq y$, thus there exists $w < x$ such
  that $\radj{f}(w) \not\leq y$. Then $y < \radj{f}(w)$ so, from
  $x \leq \meetof{f}(y) = \bigwedge_{y < y'} f(y')$, we deduce
  $x \leq f(\radj{f}(w))$. Considering that $f(\radj{f}(w)) \leq w$,
  we deduce $x \leq w$, contradicting $w < x$. Therefore,
  $\joinof{(\radj{f})}(x) \leq y$.
  
  Dually, we can argue that, for $g \in \Qm(I)$,
  $\joinof{g}(x) \leq y$ implies $x \leq \meetof{(\ladj{g})}(y)$, for
  each $g \in \LmI$.  Letting in this statement $g := \radj{f}$, we
  obtain the converse implication: $\joinof{(\radj{f})}(x) \leq y$
  implies $x \leq \meetof{(\ladj{(\radj{f})})}(y) = \meetof{f}(y)$.
  
  For the last statement, observe that the correspondence $\oppfun$ is
  order reversing since it is the composition of an order reversing
  function with a monotone one; it is an involution since
  $\oppopp{f}  = \ladj{(\JOINOF{(\meetof{(\radj{f})})})}
  =\ladj{(\radj{f})} = f$.
\end{proof}
\begin{lemma}
  \label{lemma:charOfStar}
  We have
    \begin{align}
      \label{eq:opp}
    \opp{f}(x) & = \bigvee \set{y \in I \mid f(y) < x}\,.
  \end{align}
\end{lemma}
\begin{proof}
  Recall that $\opp{f}$ has been defined as $\ladj{(\meetof{f})}$. Let
  us show that the expression on the right of equation~\eqref{eq:opp}
  yields a \LADJ for $\meetof{f}$.  For each $x,z \in I$, we have
  \begin{align*}
    \bigvee \set{y \in I \mid f(y) < x} \leq z
    & \tiff \forall y (\,f(y) < x \timplies  y \leq z \,) \\
    & \tiff \forall y (\,z < y \timplies x \leq f(y) \,) \\
    & \tiff x \leq \bigwedge_{z < y} f(y) = \meetof{f}(z)\,.
    \tag*{\qedhere}
  \end{align*}
\end{proof}

For $f,g \in \Lj(I)$, let us define
\begin{align*}
  f \otimes g & := g \circ f\,, & 1& := id_{I}
  \intertext{and, using duality as in \eqref{eq:defopluszero},}
  f \oplus g & := \OPP{(\opp{g} \otimes \opp{f})} & 0 &:= \opp{1}\,.
\end{align*}
Let us remark that the operation $\oplus$ is obtained by transporting
composition in $\Lm(I)$ to $\Lj(I)$ via the isomorphism:
\begin{align*}
  f \oplus g & = \OPP{(\opp{g} \otimes \opp{f})}
  = \opprj{(\oppml{f} \circ \oppml{g})} =
  \opprj{(\ladj{(\meetof{g} \circ \meetof{f})})} =
  \Joinof{(\meetof{g} \circ \meetof{f})}\,.
\end{align*}
In a similar way, $0$ is the image via the isomorphism of the identity
of the chain $I$, as an element of $\Qm(I)$. Using
Lemma~\ref{lemma:charOfStar}, a useful expression for $0$ is the
following:
\begin{align}
  \label{eq:zero}
  0(x) & := \bigvee_{x' < x} x'\,.
\end{align}

\begin{proposition}
  For each $f,g,h \in \Qj(I)$, $f \otimes g \leq h$ if and only if $f
  \leq h \oplus \Opp{g}$.
\end{proposition}
\begin{proof}
  Suppose $f \otimes g \leq h$, that is, $g \circ f \leq h$. We aim at
  showing that $\meetof{f} \leq \radj{g} \circ \meetof{h}$, since
  then, by applying $\joinof{(\funFiller)}$ to this relation, we shall
  obtain
  $f \leq \joinof{(\radj{g} \circ \meetof{h})} =
  \joinof{(\MEETOF{\Joinof{\radj{g}}} \circ \meetof{h})} =
  \joinof{(\MEETOF{\opp{g}} \circ \meetof{h})} = h \oplus \opp{g}$.

  This is achieved as follows. From $g(f(x)) \leq h(x)$,  for all
  $x \in I$, deduce $f(x) \leq \radj{g}(h(x))$, for all
  $x \in I$, and therefore
  \begin{align*}
    \meetof{f}(x) & = \bigwedge_{x < y}
    f(y) \leq \bigwedge_{x < y}  \radj{g}(h(y))
    =   \radj{g}(\bigwedge_{x < y} h(y)) = \radj{g}(\meetof{h}(x))\,,
  \end{align*}
  for each $x \in I$, using the fact that $\radj{g}$ is \mc.

  A similar argument, shows that if $f,g,h \in \Qm(I)$ and
  $f \leq g \circ h$, then
  $\ladj{g} \circ \joinof{f} \leq \joinof{h}$. For $f,g,h \in \Qj(I)$,
  this yields that $f \leq h \oplus g$ implies
  $f \otimes \opp{g} \leq h$. Therefore, if $f \leq h \oplus \opp{g}$,
  then $f \otimes g =f \otimes \oppopp{g} \leq h$.
\end{proof}

\begin{corollary}
  \label{cor:quantalesFromChains}
  For each perfect chain $I$ the  \rl $\Qj(I)$ of \jc functions from
  $I$ to itself is a \msaq.
\end{corollary}
\begin{proof}
  By the previous Lemma and by Lemma~\ref{lemma:verifHalf}, the
  antitone involution $\oppfun$ yields the dual operation $\oplus$
  satisfying the residuation relations~\eqref{eq:residuationoplus}. By
  equation \eqref{eq:zero}, it is also clear that the relation
  $0 \leq 1$, so the mix rule holds in $\Qj(I)$.
\end{proof}

\begin{remark}
  The \saq structure on $\Qj(\In)$ is the unique possible one.  It was
  shown in \cite[\S 4.1]{MSRSDLL} using duality theory that dualizing
  objects of in $\Qj(\In)$ are in bijection with isomorphisms of the
  ordered set $\set{1,\ldots ,n}$. Obviously, there is just one such
  isomorphism.  On the other hand, the dualizing elements of an
  involutive residuated lattice such that $1 = 0$ are exactly the
  elements $f$ that are invertible (in particular, this is the case
  for the quantale $\QjI$). We sketch a proof of this.  If $f$ is
  dualizing, then
  $1 = f \rlimpl (1 \lrimpl f) = f \rlimpl f = f \oplus
  \opp{f}$. Similarly, $1 = \opp{f} \oplus f$ and, dually,
  $1 = f \otimes \opp{f} = \opp{f} \otimes f$.  Vice versa, if $f$ has
  an inverse $f^{-1}$, then
  $f \oplus g = f \otimes f^{-1} \otimes (f \oplus g) \leq f \otimes
  (((f^{-1} \otimes f) \oplus g)) = f \otimes g$, so
  $f \otimes g = f \oplus g$, for any $g$. Then $f^{-1} = \opp{f}$,
  since $1 \leq f \oplus \opp{f} = f \otimes \opp{f} \leq 0 =1$, and
  $f \rlimpl (g \lrimpl f) = f \oplus (\opp{f} \otimes g) = f \otimes
  \opp{f} \otimes g = g$.
\end{remark}

\section{Lattices from mix \Lbs{s}}
\label{sec:latticesFromQuantales}

In this section $d$ shall be a fixed integer greater than or equal to
$2$ (the case $d = 2$ being trivial).
Given an \lbs $Q$, consider the product
$\prodd{Q} := \prod_{1 \leq i < j \leq d} Q$.
We say that a tuple
$f = \langle f_{i,j} \mid 1 \leq i < j \leq d\rangle$ of this product
is \emph{closed} (\resp \emph{open}) if
\begin{align*}
  f_{i,j} \otimes f_{j,k} & \leq f_{i,k}
  \qquad
  (\resp   \,f_{i,k} \leq f_{i,j} \oplus f_{j,k}\,)
  \,.
\end{align*}
Recall that $\prodd{Q}$ has a lattice structure induced by the
coordinate-wise meets and joins.  It is then easily verified that
closed tuples are closed under arbitrary meets and open tuples are
closed under arbitrary joins.
\begin{remark}
  If $Q$ is an \irl, then $f$ is closed if and only if
  $\opp{f} := \langle \opp{(f_{\sigma(j),\sigma(i)})} \mid 1 \leq i <
  j \leq d\rangle$ is open, where $\sigma(i) := d - i + 1$, for each
  $i \in [d]$. In this case the correspondence sending $f$ to
  $\opp{f}$ is an antitone involution of $\PrdQ$, sending closed
  tuples to open ones, and vice versa.
\end{remark}

For $(i,j) \in \cd$, a subdivision of the interval $[i,j]$ is a subset
of this interval containing the endpoints $i$ and $j$.  We write
such a subdivision as sequence of the form
$i = \ell_{0} < \ell_{1} < \ldots \ell_{k -1}< \ell_{k} = j$ with
$i < \ell_{i} < j$, for $i = 1,\ldots ,k -1$.
We shall use then $S_{i,j}$ to denote the set of subdivisions of the
interval $[i,j]$.
\begin{lemma}
  \label{lemma:ClosureInterior}
  For each $f \in \prodd{Q}$, the tuple $\closure{f}$ defined by
  \begin{align*}
    \closure{f}_{i,j} & := \bigvee_{i < \ell_{1} < \ldots \ell_{k -1}<
      j \in S_{i,j}} f_{i,\ell_{1}}\otimes
    f_{\ell_{1},\ell_{2}}\otimes \ldots \otimes f_{\ell_{k-1},j}\,.
  \end{align*}
  is the least closed tuple $g$ such that $f \leq g$.
  Dually, if we set
  \begin{align*}
    \interior{f}_{i,j} & := \bigwedge_{i < \ell_{1} < \ldots \ell_{k
        -1}< j \in S_{i,j}} f_{i,\ell_{1}}\oplus
    f_{\ell_{1},\ell_{2}}\oplus \ldots \oplus
    f_{\ell_{k-1},j}\,.
  \end{align*}
  then 
  $\interior{f}$ is the greatest open tuple below $f$.
\end{lemma}
\begin{proof}
  It suffices to prove the first statement. Since
  $\set{i < j} \in S_{i,j}$, then $f_{i,j} \leq \closure{f}_{i,j}$,
  for each $(i,j) \in \cd$, thus $f \leq \closure{f}$.  Now, if $g$ is
  closed and $f \leq g$, then, for each subdivision
  $i < \ell_{1} < \ldots \ell_{k-1} < j$, we have
  \begin{align*}
    f_{i,\ell_{1}}\otimes \ldots \otimes f_{\ell_{k-1},j}
    & \leq g_{i,\ell_{1}}\otimes \ldots \otimes g_{\ell_{k-1},j}
    \leq g_{i,j}\,.
  \end{align*}
  We are left to prove that $\closure{f}$ is closed. 
  To the sake of being concise, if $\varsigma \in S_{i,j}$ is
  $i = \ell_{0} < \ell_{1} < \ldots \ell_{k -1}< \ell_{k} = j$, then
  we let $\pi(f,\varsigma)$ be
  $f_{i,\ell_{1}}\otimes f_{\ell_{1},\ell_{2}}\otimes \ldots \otimes
  f_{\ell_{k-1},j}$. Observe next that if $\varsigma \in S_{i,j}$ and
  $\varsigma' \in S_{j,k}$, then the set theoretic union
  $\varsigma \cup \varsigma'$ belongs to $S_{i,k}$ and, moreover,
  $\pi(\varsigma,f) \otimes \pi(\varsigma',f) = \pi(\varsigma \cup
  \varsigma',f)$. We have therefore
  \begin{align*}
    \closure{f}_{i,j} \otimes \closure{f}_{j,k} & =
    \bigvee_{\varsigma \in S_{i,j}} \pi(\varsigma,f)\otimes
    \bigvee_{\varsigma' \in S_{j,k}} \pi(\varsigma',f)
    = \bigvee_{\varsigma \in S_{i,j},
      \varsigma'  \in S_{j,k}
    }     \pi(\varsigma,f) \otimes \pi(\varsigma',f)
    \\
    &
    =
     \bigvee_{\varsigma \in S_{i,j},
      \varsigma'  \in S_{j,k}
    }     \pi(\varsigma \cup \varsigma',f)
    \leq
    \bigvee_{\varsigma'' \in S_{i,k}} \pi(\varsigma'',f)
    = \closure{f}_{i,k}\,.
    \tag*{\qedhere}
  \end{align*}
\end{proof}
We call the map $f \mapsto \closure{f}$ the \emph{closure}, and the
map $f \mapsto \interior{f}$ the \emph{interior}.  Then a tuple is
closed if and only of it is equal to its closure, and a tuple is open
if and only of it is equal to its interior.  We shall be interested in
tuples $f \in \PrdQ$ that are \emph{clopen}, that is, they are at the
same time closed and open.

\begin{proposition}
  \label{prop:interiorofclosed}
  Let $Q$ be a mix \lbs and let $f \in \PrLQ$. If $f$ is closed, then so
  is $\interior{f}$.
\end{proposition}
\begin{proof}
  Let $i,j,k \in [d]$ with $i < j < k$.  We need to show that
  \begin{align*}
    \interior{f}_{i,j} \otimes \interior{f}_{j,k} & \leq
    f_{i,\ell_{1}}\oplus \ldots \oplus f_{\ell_{n-1},k}
  \end{align*}
  whenever $i < \ell_{1} < \ldots \ell_{n -1}< k \in S_{i,k}$.  This
  is achieved as follows. Let $u \in \set{0,1,\ldots ,n-1}$ be such
  that $j \in [\ell_{u}, \ell_{u +1})$. Firstly suppose that $\ell_{u}
  < j$; put then
  \begin{align*}
    \alpha & := f_{i,\ell_{1}} \oplus \ldots \oplus f_{\ell_{u-1},\ell_{u}}\,, &
    \delta & := f_{\ell_{u+1}} \oplus \ldots \oplus f_{\ell_{n-1},k} \, \\
    \beta & := f_{\ell_{u},j} \, &
    \gamma & := f_{j,\ell_{u +1}}\,.
  \end{align*}
  Then
  \begin{align*}
    \int{f}_{i,j} \otimes \int{f}_{j,k} & \leq (\alpha \oplus \beta)
    \otimes (\gamma \oplus \delta)\,, \tag*{by definition of
      $ \int{f}_{i,j}$ and $\int{f}_{j,k}$,}
    \\
    & \leq \alpha \oplus (\beta \otimes \gamma) \oplus \delta\,,
    \tag*{by the inequation~\eqref{eq:hemidistr1},}
    \\
    & \leq \alpha \oplus f_{\ell_{u},\ell_{u+1}} \oplus \delta
    = f_{i,\ell_{1}}\oplus
    \ldots \oplus f_{\ell_{n-1},k}\,,
    \tag*{since $f$ is closed.}
  \end{align*}
  Notice that we might have that $\alpha$ defined above is an empty
  (co)product (e.g. when $u = 0$), in which case we can use the
  inclusion \eqref{eq:hemidistr1p} in place of
  \eqref{eq:hemidistr1}. A similar remark has to be raised when
  $\delta$ defined above is an empty (co)product (when $u = n-1$), in
  which case we use inclusion \eqref{eq:hemidistr2p}.  Finally, if 
  $j = \ell_{u}$, then let $\alpha,\gamma,\delta$ as above, we derive
  \begin{align*}
    \int{f}_{i,j} \otimes \int{f}_{j,k} &
    \leq \alpha \otimes (\gamma \oplus \delta)
    \leq \alpha \oplus (\gamma \oplus \delta)
    = f_{i,\ell_{1}}\oplus
    \ldots \oplus f_{\ell_{n-1},k}\,,
  \end{align*}
  using the mix rule \eqref{eq:mix}. 
\end{proof}

Since the definition of \lbs is auto-dual, we also have the following
statement:
\begin{proposition}
  \label{prop:closureofopen}
  Let $Q$ be a mix \lbs and let $f \in \PrLQ$. If $f$ is open, then so
  is $\closure{f}$.
\end{proposition}

\begin{definition}
  For $Q$ a mix \lbs, $\Ld{Q}$ shall denote the set of clopen tuples
  of $\PrLQ$.
\end{definition}

\begin{theorem}
  The set $\Ld{Q}$  is, with the ordering inherited from $\PrdQ$, a lattice.
  \label{theo:meetsAndJoin}
\end{theorem}
\begin{proof}
  For a family $\set{f_{i} \mid i \in I}$, with each
  $f_{i}$ clopen, define 
  \begin{align}
    \bigvee_{\Ld{Q}} \set{ f_{i} \mid i \in I } & :=
    \cl{\bigvee \set{ f_{i} \mid i \in I } }\,, & 
    \bigwedge_{\Ld{Q}} \set{ f_{i} \mid i \in I } & :=
    \int{(\bigwedge \set{ f_{i} \mid i \in I } )}\,,
    \label{eq:defJoinAndMeet}
  \end{align}
  whenever the supremum $\bigvee \set{ f_{i} \mid i \in I }$ (\resp
  infimum $\bigwedge \set{ f_{i} \mid i \in I }$) exists in
  $\PrdQ$. Since this join (\resp meet) is open, its closure is clopen
  by Proposition~\ref{prop:closureofopen} (\resp
  Proposition~\ref{prop:interiorofclosed}) and therefore it belongs to
  $\Ld{Q}$.  Then It is easily seen that this is the supremum (\resp
  infimum) of the family $\set{f_{i} \mid i \in I}$ in the poset
  $\Ld{Q}$.
\end{proof}

\begin{example}
  Let $Q = \two$ be the two element Boolean algebra $\two$. We
  identify a tuple $\chi \in \prodd{\two}$ with the characteristic map
  of a subset $S_{\chi}$ of $\set{(i,j) \mid 1 \leq i < j \leq
    d}$. Think of this subset as a relation. Then $\chi$ is clopen if
  both $S_{\chi}$ and its complement in
  $\set{(i,j) \mid 1 \leq i < j \leq d}$ are transitive
  relations. These subsets are in bijection with permutations of the
  set $[d]$, see \cite{STA0}; the lattice $\Ld{\two}$ is therefore
  isomorphic to the well-known \Permutohedron, aka the \wBo.
\end{example}
\begin{example}
  On the other hand, if $Q$ is the Sugihara monoid on the
  three-element chain described in Example~\ref{ex:Sugihara}, then the
  lattice of clopen tuples is isomorphic to the lattice of
  pseudo-permutations, see \cite{KLNPS01,STA,facialWeakOrder}.
\end{example}

\begin{example}
  Let us consider a finite chain $\I_{n} =\set{0,\ldots ,n}$ and the
  quantale $\Qj(\I_{n})$. Let $d \cdot n$ be the integer vector of
  length $d$ whose all entries are equal to $n$. We claim that the
  lattice $\Ld{\Qj(\I_{n})}$ is isomorphic to the multinomial lattice
  $\L(d \cdot n)$ of \cite{BB}, see also \cite[\S 8-10]{STA}. It is
  argued in \cite{STA} that elements of these multinomial lattices are
  in bijection with some clopen tuples of the product
  $\prodd{\L(n,n)}$.  Considering that a binomial lattice $\L(n,n)$ is
  isomorphic (as a lattice) to the quantale $\Qj(\I_{n})$, we are left
  to verify that the two notions of closed/open tuple coincide via the
  bijection.

  For $x,y \in \setn$, let $\langle x,y\rangle$ denote the least \jc
  function $f \in \Qj(\I_{n})$ such that $y \leq f(x)$. Elements of
  the form $\langle x,y\rangle$ are the \jp elements of $\Qj(\I_{n})$.
  A tuple $f \in \prodd{\Qj(\I_{n})}$ is called closed in \cite{STA}
  if, for each $i,j,k$ with $1 \leq i < j < k \leq d$ and each triple
  $x,y,z \in \setn$, $\langle x,y\rangle \leq f_{i,j}$ and
  $\langle y,z\rangle \leq f_{j,k}$ imply
  $\langle x, z\rangle \leq f_{i,k}$.  Considering that, for
  $f \in \Qj(\I_{n})$, $\langle x,y\rangle \leq f$ if and only if
  $y \leq f(x)$, closedness is easily seen to be equivalent to the
  condition $f_{j,k} \circ f_{i,j} \leq f_{i,k}$, that is, to the
  notion of closedness introduced in this section.  Let us argue that
  a tuple is open as defined in \cite{STA} if and only if it is open
  as defined in this section.
  To this goal, for $x,y \in \setn$, let $\mirrd{x,y}$ be the greatest
  \jc function $f \in \Qj(\In)$ such that $f(x) \leq y -1$. Elements
  of the form $\mirrd{x,y}$ are the \mirr elements of $\Qj(\I_{n})$
  and, moreover, $\opp{\mirrd{x,y}} =\jirrd{y,x}$. In \cite{STA} a
  tuple $f \in \prodd{\Qj(\I_{n})}$ is said to be open if
  $f_{i,j} \leq \mirrd{x,y} $ and $f_{j,k} \leq \mirrd{y,z}$ imply
  $f_{i,k} \leq \mirrd{x,z}$, for each $x,y,z \in \setn$ and whenever
  $1 \leq i < j < k \leq d$.  This condition is equivalent to
  $\jirrd{y,x} = \opp{\mirrd{x,y}} \leq \opp{f_{i,j}}$ and
  $\jirrd{z,y} = \opp{\mirrd{y,z}} \leq \opp{f_{j,k}}$ imply
  $\jirrd{z,x} = \opp{\mirrd{x,z}} \leq \opp{f_{i,k}}$, for each
  $z,y,x \in \setn$ and $1\leq i < j < k \leq d$. As before, this is
  equivalent to
  $\opp{f_{j,k}} \otimes \opp{f_{i,j}} \leq \opp{f_{i,k}}$ and then to
  $f_{i,k} \leq \opp{(\opp{f_{j,k}} \otimes \opp{f_{i,j}})} =
  f_{i,j}\oplus f_{j,k}$, yielding the notion of openness as defined
  here.
 \end{example}

\begin{proposition}
  \label{prop:functor}
  $\Ld{-}$ is a limit-preserving functor from the category of
  mix \lbs{s} to the category of
  lattices. 
\end{proposition}
\begin{proof}
  Let $\psi : Q_{0} \rto Q_{1}$ be an \lbs morphism (that is, a
  lattice morphism which, moreover, preserves $\otimes$ and
  $\oplus$). The map $\prodd{\psi} : \prodd{Q_{0}} \rto \prodd{Q_{1}}$
  defined by $[\prodd{\psi}(f)]_{i,j} := \psi(f_{i,j})$ commutes both
  with the closure map and with the interior map, since these maps are
  defined by means of the operations preserved by $\psi$.
  Consequently, the image by $\prodd{\psi}$ of a clopen is clopen.
  Similarly, the lattice operations on clopens, defined in
  equation~\eqref{eq:defJoinAndMeet}, are preserved by $\prodd{\psi}$
  since (for example for the joins) this function preserves the joins
  of $\prodd{Q_{0}}$ and the closure.

  Since the forgetful functor from the category of lattices to the
  category of sets creates limits, in order to argue that the functor
  $\Ld{-}$ preserves limits, we can consider it as a functor form the
  category of mix \lbs{s} to the category of sets and functions and
  show that it preserves limits.

  Let $\mathcal{C}$ be the category of \lbs{s} and their morphisms and
  consider the category of limit preserving functors from
  $\mathcal{C}$ to the category $\mathcal{S}$ of sets and functions.
  This category contains the forgetful functor (that we note here $X$)
  and is closed under limits. This holds since limits in the category
  of functors from $\mathcal{C}$ to $\mathcal{S}$ are computed
  pointwise.
  It is then enough to observe that
  $\Ld{-} : \mathcal{C} \rto \mathcal{S}$ is the following equalizer:
  \begin{center}
    \hfill
    \begin{tikzcd}
      \Ld{X} \ar[rr, hook]
      &&
      \prodd{X} \ar[rr,shift left =4, "\interior{(-)}"]
      \ar[rr,"id"  description]
      \ar[rr, shift right =4, "\closure{(-)}"']
      & &\prodd{X} \\[6mm]
    \end{tikzcd}
    \qedhere
  \end{center}
\end{proof}
In particular, from the previous proposition we obtain the following
statement, that we shall use in Section~\ref{sec:embeddings}.
\begin{proposition}
  If $i : Q_{0} \rto Q_{1}$ is an injective homomorphism of mix
  \lbs{s}, then $\Ld{i} : \Ld{Q_{0}} \rto \Ld{Q_{1}}$ is an 
  embedding.
\end{proposition}

The goal of the rest of this section is to argue that clopen tuples
naturally arise as some sort of enrichment (in the sense of
\cite{lawvere,kelly,stubbe}) or metric of a set $X$. For the sake of
this discussion, we shall fix an \irl $Q$ with the property that
$0 = 1$. This equality holds in the quantale $\LjI$ studied in
Section~\ref{sec:quantaleLI}, but fails in other mix \irl{s}, e.g. in
the quantales $\Qj(\In)$.

A \emph{skew metric} of $X$ over $Q$ is a map
$\delta : X \times X \rto Q$ such that, for all $x,y,z \in X$,
\begin{align*}
  \delta(x,x) & \leq 0  \,,\\
  \delta(x,z) & \leq \delta(x,y)\oplus \delta(y,z)\,, \\
  \delta(x,y) & = \opp{\delta(y,x)}\,. 
\end{align*}
That is, a skew metric is a semi-metric (see e.g. \cite{weakmetrics})
with values in $Q$, where the symmetry condition has been replaced by
the last requirement, skewness. Similar kind of metrics have been
considered in the literature, for example in \cite{pouzet}.
Observe that (when $X \neq \emptyset$)
$1 = \opp{0} \leq \opp{\delta(x,x)} = \delta(x,x) \leq 0$, so if $Q$
is mix, then necessarily $1 =0$.

\begin{lemma}
  \label{lemma:defskewmetric}
  Suppose in $Q$ the equality $1 = 0$ holds.  By defining
  \begin{align*}
    \delta_{f}(i,j) & :=
    \begin{cases}
      f_{i,j}\,, & i < j\,, \\
      0\,, & i = j\,, \\
      \Opp{f_{j,i}}\,, & j < i\,,
    \end{cases}
  \end{align*}
  every clopen tuple $f$ of $\prodd{Q}$ yields a unique skew metric on
  the set $\setof{d}$. Every skew metric on the set $\setof{d}$ with
  values in $Q$ arises in this way.
\end{lemma}

The Lemma is an immediate consequence of the following statement:
\begin{lemma}
  \label{lemma:enrichment}
  A tuple $f$ is clopen if and only if $\delta_{f}$ is a skew metric.
\end{lemma}
\begin{proof}
  Suppose $\delta_{f}$ is a skew metric.  For
  $1 \leq i < j < k \leq d$, we have
  $f_{i,k} \leq f_{i,j}\oplus f_{j,k}$ (openness) and
  $f_{k,i} \leq f_{k,j}\oplus f_{j,i}$ which in turn is equivalent to
  $ f_{i,j} \otimes f_{j,k} \leq f_{i,k}$ (closedness).

  Conversely, suppose that $f$ is clopen. Say that the pattern $(ijk)$
  is satisfied by $f$ if $f_{i,k} \leq f_{i,j}\oplus f_{i,j}$. If
  $\card(\set{i,j,k}) \leq 2$, then $f$ satisfies the pattern $(ijk)$
  if $i=j$ or $j = k$, since then $f_{i,j} = 0$ or $f_{j,k} = 0$. If
  $i = k$, then $0 \leq f_{i,j} \oplus f_{j,i}$ is equivalent to
  $f_{i,j} \leq f_{i,j}$.

  Suppose therefore that
  $\card(\set{i,j,k}) = 3$.
  By assumption, $f$ satisfies $(ijk)$ and $(kji)$ whenever
  $i < j < k$.  Then it is possible to argue that all the patterns on
  the set $\set{i,j,k}$ are satisfied by observing that if $(ijk)$ is
  satisfied, then $(jki)$ is satisfied as well: from
  $f_{i,k} \leq f_{i,j}\oplus f_{j,k}$, derive
  $f_{i,k}\otimes f_{k,j} = f_{i,k}\otimes \Opp{f_{j,k}} \leq f_{i,j}$
  and then
  $f_{j,i} = \Opp{f_{i,j}}\leq \Opp{(f_{i,k}\otimes f_{k,j})} = f_{j,k}
  \oplus f_{k,i}$.
\end{proof}
\begin{remark}
  \label{rem:suffCondClopen}
  In the next sections we shall often need to verify that some tuple
  $f \in \prodd{Q}$ is clopen. A simple sufficient condition is that,
  for each $i,j,k \in \setof{d}$ with $i < j < k$, either
  $f_{i,k} = f_{i,j} \otimes f_{j,k}$, or
  $f_{i,k} = f_{i,j} \oplus f_{j,k}$.  Indeed, from
  $f_{i,k} = f_{i,j} \otimes f_{j,k}$ we derive
  $f_{i,j} \otimes f_{j,k} \leq f_{i,k} = f_{i,j} \otimes f_{j,k} \leq
  f_{i,j} \oplus f_{j,k}$, using the mix rule.  Similarly,
  $f_{i,k} = f_{i,j} \oplus f_{j,k}$ implies
  $f_{i,j} \otimes f_{j,k} \leq f_{i,k} \leq f_{i,j} \oplus f_{j,k}$.
\end{remark}

\section{The \msaq $\LjI$}
\label{sec:quantaleLI}

From this section onward $\I$ denotes the unit interval of the reals,
$\I := [0,1]$. Recall that we use $\LjI$ for the set of \jcont
functions from $\I$ to itself.  Notice that a monotone function
$f : \I \rTo \I$ is \jcont if and only if
\begin{align}
  \label{eq:joincontinuous}
  f(x) & = \bigvee_{y < x} f(y)\,, \quad \text{or
    even} \quad f(x) = \bigvee_{y < x, \,y \in \I \cap \Q} f(y)\,,
\end{align}
see Proposition~2.1, Chapter II of \cite{compendiumCL}.
According to Example~\ref{ex:jcfunctions},
we have:
\begin{lemma}
  Composition induces a quantale structure on $\LjI$.
\end{lemma}
Let now $\LmI$ denote the collection of \mcont functions from
$\I$ to itself.
By duality, we obtain:
\begin{lemma}
  Composition induces a dual quantale structure on $\LmI$.
\end{lemma}
With the next set of observations we shall see $\LjI$ and $\LmI$ are
order isomorphic. For a monotone function $f : \I \rto \I$, define
\begin{align*}
  \meetof{f}(x) & = \bigwedge_{x < x'} f(x')\,, &
  \joinof{f}(x) & = \bigvee_{x' < x} f(x')\,.
\end{align*}
\begin{lemma}
  \label{lemma:meetofbelowjoinof}
  Let $f : \I \rto \I$ be monotone. If $x < y$, then
  $\meetof{f}(x) \leq \joinof{f}(y)$.
\end{lemma}
\begin{proof}
  Pick $z \in \I$ such that $x < z < y$ and observe then that
  $\meetof{f}(x) \leq f(z) \leq \joinof{f}(y)$.
\end{proof}
\begin{proposition}
  \label{prop:meet-cont-closure}
  For a monotone $f : \I \rto \I$, the following statements hold:
  \begin{enumerate}
  \item   $\meetof{f}$ is the least \mc function above $f$
    and $\joinof{f}$ is the greatest \jcont function below $f$,
  \item the relations $\meetofjoinof{f} = \meetof{f}$ and
    $\joinofmeetof{f} = \joinof{f}$ hold,
  \item  the
    operations $\joinof{(\,\cdot\,)} : \LmI \rto \LjI$ and
    $\meetof{(\,\cdot\,)} : \LjI \rto \LmI$ are inverse order
    preserving bijections.
  \end{enumerate}
\end{proposition}
\begin{proof}
  (1) We only prove the first statement. Let us show that $\meetof{f}$
  is \mc; to this goal, we use equation~\eqref{eq:joincontinuous}:
  \begin{align*}
    \bigwedge_{x < t} \meetof{f}(t)
    & = \bigwedge_{x < t} \bigwedge_{t < t'} f(t')
    = \bigwedge_{x < t} f(t') 
    = \meetof{f}(x)\,.
  \end{align*}
  We observe next that $f \leq \meetof{f}$, as if $x < t$, then
  $f(x) \leq f(t)$.  This implies that if $g \in \LmI$ and
  $\meetof{f} \leq g$, then $f \leq \meetof{f} \leq g$.
  Conversely, if $g \in \LmI$ and $f \leq g$, then
  \begin{align*}
    \meetof{f}(x)  & = \bigwedge_{x < t} f(t)
    \leq  \bigwedge_{x < t} g(t) = g(x)\,.
  \end{align*}
  Let us prove (2) and (3).  Clearly, both maps are order
  preserving. Let us show that $\meetofjoinof{f} = \meetof{f}$
  whenever $f$ is order preserving.
  We have $\meetofjoinof{f} \leq \meetof{f}$, since
  $\joinof{f} \leq f$ and $\meetof{(-)}$ is order preserves the
  pointwise ordering. For the converse inclusion, recall from the
  previous lemma that if $x < y$, then
  $\meetof{f}(x) \leq \joinof{f}(y)$, so
  \begin{align*}
    \meetof{f}(x) & \leq \bigwedge_{x < y} \joinof{f}(y) =  \meetofjoinof{f}(x)\,,
  \end{align*}
  for each $x \in \I$. Finally, to see that $\meetof{(-)}$ and
  $\joinof{(-)}$ are inverse to each other, observe that of
  $f \in \LmI$, then $\meetofjoinof{f} = \meetof{f} =f$.
  The equality $\joinofmeetof{f} =f$ for $f \in \LjI$ is derived similarly.
\end{proof}

\begin{corollary}
  \label{cor:distrlattice}
  $\LjI$ is a complete distributive lattice.
\end{corollary}
\begin{proof}
  The interval $\I$ is a complete distributive lattice, whence the set
  $\I^{\I}$ of all functions from $\I$ to $\I$, is also a complete
  distributive lattice, under the pointwise ordering and the pointwise
  operations. The subset of monotone functions from $\I$ to $\I$ is
  closed under infs and sups from $\I^{\I}$. In view of
  Proposition~\ref{prop:meet-cont-closure}, \jc functions are the
  monotone functions that are fixed points of the interior operator
  $f \mapsto \joinof{f}$. As from standard theory, it follows that
  $\LjI$ is a complete lattice, that \jc functions are closed under
  pointwise suprema, and that infima in $\LjI$ are computed as
  follows:
  \begin{align*}
    (\bigwedge_{i \in I} f_{i})(x)
    & = \bigvee_{y < x} \inf \set{f_{i}(y)\mid i \in I }\,.
  \end{align*}
  Finally notice  that, in case $I = \set{1,2}$, then
  \begin{align*}
    (f_{1} \land f_{2})(x)
    & = \bigvee_{y < x} \min(f_{1}(y),f_{2}(y)) \\
    & = \bigvee_{y_{1} < x} \bigvee_{y_{2} < x}
    \min(f_{1}(y_{1}),f_{2}(y_{2}))\,, \tag*{since the set
      $\set{y \in \I \mid y < x}$ is upward directed,}
    \\
    &= \min(\bigvee_{y_{1} < x} f_{1}(y_{1}),\bigvee_{y_{2} <
      x}f_{2}(y_{2})) = \min(f_{1}(x),f_{2}(x))\,,
  \end{align*}
  where the last step follows from $\joinof{f_{i}} = f_{i}$,
  $i \in \set{1,2}$.  Therefore finite (non-empty) meets are computed
  pointwise, and this implies that $\LjI$ is a distributive lattice.
\end{proof}

Considering that $\I$ is a complete lattice,
Proposition~\ref{prop:meet-cont-closure} shows that it is also a
perfect chain and therefore. According to
Corollary~\ref{cor:quantalesFromChains}, we deduce the following
statement.
\begin{corollary}
  $\QjI$ is a \msaq.
\end{corollary}

\section{Paths}
\label{sec:paths}

Let in the following $d \geq 2$ be a fixed integer; we shall use
$\I^{d}$ to denote the $d$-fold product of $\I$ with itself. That is,
$\I^{d}$ is the usual geometric cube in dimension $d$. Let us recall
that $\I^{d}$, as a product of the poset $\I$, has itself the
structure of a poset (the order being coordinate-wise) and, moreover,
of a complete lattice.

\begin{definition}
  A \emph{path in $\I^{d}$} is a chain $C \subseteq \I^{d}$ with the
  following properties:
  \begin{enumerate}
  \item if $X \subseteq C$, then $\bigwedge X \in C$ and
    $\bigvee X \in C$,
  \item $C$ is dense as an ordered set: if $x,y \in C$ and $x < y$,
    then $x < z < y$ for some $z \in C$.
  \end{enumerate}
\end{definition}
That is, we have defined a path in $\I^{d}$ as a totally ordered dense
sub-complete-lattice of $\I^{d}$. We are going to see that paths in
$\Id$ can be characterized in many ways.

\begin{lemma}
  \label{lemma:pathmaximal}
  Paths in $\I^{d}$ are exactly the maximal chains of the poset
  $\I^{d}$.
\end{lemma}
\begin{proof}
  We firstly argue that every path in $\I^{d}$ is a maximal chain of
  $\I^{d}$.

  Let $C \subseteq \I^{d}$ be a path and suppose that there exists
  $z \in \I^{d} \setminus C$ such that $C \cup \set{z}$ is a
  chain. Let $z^{-} = \set{c \in C \mid c < z}$ and
  $z^{+} = \set{c \in C \mid z < c}$. Since $z \not \in C$ and $C$ is
  closed under meets and joins, we have
  $\bigvee z^{-} < z < \bigwedge z^{+}$, with
  $\bigvee z^{-},\bigwedge z^{+} \in C$. By density, let $w \in C$ be
  such that $\bigvee z^{-} < w < \bigwedge z^{+}$. Since
  $w \in C \subseteq C \cup \set{z}$ and the latter is a chain, then
  $w < z$ or $z < w$. In the first case we obtain
  $w \leq \bigvee z^{-}$ and in the second case
  $\bigwedge z^{+} \leq w$ and, in both cases, we have a
  contradiction.

  Next, we argue that every maximal chain of $\Id$ is a path in $\Id$.
  Let $C$ be a maximal chain of $\Id$.  Take $X\subseteq C$ and let
  $a:=\bigwedge X\in \Id$. The maximality of $C$ implies that
  $0,1\in C$ and so $a\in C$ whenever $X=\emptyset$ or $X=C$. Suppose
  that $X\neq \emptyset$. We claim that $C\cup \set{a}$ is a chain and
  consequently $a\in C$ by the maximality of $C$. Let $c\in C$; if
  $c\nleqslant a$ then $c\nleqslant x$, for some $x\in X$, which
  implies $x<c$ and so $a<c$; if $a\nleqslant c$, then $x\nleqslant c$
  for every $x\in X$, which implies $c<x$ for every $x\in X$, and so
  $c\leqslant a$.  Thus $C\cup\set{a}$ is a chain as aimed.  Let us
  now prove that $C$ is dense.  Let $x<y$ in $C$. Suppose that for
  every $c\in C$ we have $y\leqslant c$ or $c\leqslant x$.  Since
  $x<y$, there exists $j\in [d]$ such that $x_j<y_j$. The density of
  $\I$ implies the existence of $z_j\in I$ such that
  $x_j<z_j<y_j$. Take $w\in \I^d$ to be defined by $w_j=z_j$ and
  $w_i=x_i$ for $i\neq j$. Clearly $x<w<y$. If $w\notin C$, then
  $C\cup\set{w}$ is not a chain and there exists $c\in C$ such that
  $w\nleqslant c$ and $c\nleqslant w$; consequently, $y\nleqslant c$
  and $c\nleqslant x$, which contradicts the assumption that
  $y\leqslant c$ or $c\leqslant x$, for each $c \in C$.  Thus there
  must be $c\in C$ such that $x<c<y$.
\end{proof}

We carry over with a characterization of maximal chains of $\I^{d}$
which justifies naming them paths.
\begin{lemma}
  A monotone function $\p : \I \rto \Id$ such that $\p(0) = \vec{0}$
  and $\p(1) = \vec{1}$ is topologically continuous if and only if it
  is \bcont. Consequently, its image in $\Id$ is a path.
\end{lemma}
\begin{proof}
  Let $\p$ be as in the statement of the Lemma.  For each
  $i \in \set{1,\ldots ,d}$, let $\pi_{i} : \Id \rto \I$ be the
  projection on the $i$-th coordinate, and set
  $f_{i } := \pi_{i} \circ f$, so each $f_{i}$ is monotone.  Recall
  the standard theorem on existence/characterization of left limits of
  monotone functions:
  $\lim_{y \to x^{-}}f_{i}(y) = \bigvee f_{i}([0,x))$.
  
  If $p$ is \tcont, then each $f_{i}$ is \tcont. Let $X \subseteq \I$
  and observe that $X$ is cofinal in $[0,\bigvee X)$ (that is, for
  each $y \in [0,\bigvee X)$ here exists $x \in X$ such that
  $y \leq x$. This implies that
  $\bigvee g([0,\bigvee X)) \leq \bigvee g(X)$, for each monotone
  function $g$.
  It follows that
  \begin{align*}
    f_{i}(\bigvee X) & = \lim_{y \to
      (\bigvee x)^{-}}f_{i}(y)\,, \tag*{since $f_{i}$ is \tcont,} \\
    &= \bigvee f_{i}( [0,\bigvee X)) 
    \leq \bigvee f_{i}(X)\,.
  \end{align*}
  Since the opposite inclusion holds by monotonicity, this shows that
  each $f_{i}$ is \jcont, so $f$ is \jcont. In a similar way, $f$ is
  \mcont.

  Conversely, let us suppose that $f$ is \bcont. Thus, for each $x \in
  \I$, we have
  \begin{align*}
    \lim_{y \to x^{-}}f_{i}(y) & = \bigvee f_{i}([0,x)) 
    = f(x) = \bigwedge f_{i}((x,1])
    =  \lim_{z \to x^{+}}f_{i}(z)\,,
  \end{align*}
  showing that each $f_{i}$ (and therefore $f$) is \tcont.

  For the last statement, let $C = \p(\I)$. Let $X \subseteq C$ and
  $Y \subseteq \I$ be such that $\p(Y) = X$. Then
  $\bigvee X = \bigvee \p(Y) = \p(\bigvee Y) \in C$; in a similar way,
  $\bigwedge Y \in C$. Let us show that $C$ is dense. Let $x,y\in \I$
  be such that $\p(x) < \p(y)$.  Since $\p$ is monotone, we also have
  $x < y$ (use Lemma~\ref{lemma:reflect}).  Consider then the image of
  the connected interval $[x,y]$. Since $\p$ is \tcont, its image
  cannot be the disconnected two points set
  $\set{\p(x),\p(y)}$. Therefore there exists $z \in (x,y)$ such that
  $\p(z) \not\in \set{\p(x),\p(y)}$; then, by monotonicity, we get
  $\p(x) < \p(z) < \p(y)$.
\end{proof}
Thus, if $\p : \I \rto \I^{d}$ is a monotone \tcont function with
$\p(0) = \vec{0}$ and $\p(1) = \vec{1}$, then
$\p(\I) \subseteq \I^{d}$ is a path.  We are going to show that every
path arises in this way.

\begin{lemma}
  \label{lemma:reflect}
  Consider a monotone function $f : C \rTo P$ where $C$ is a chain and
  $P$ is any poset.  Then $f$ reflects the strict order: $f(x) < f(y)$
  implies $x < y$.
\end{lemma}
\begin{proof}
  Suppose $f(x) < f(y)$. We have $y \leq x$ or $x < y$. However, if
  $y \leq x$, then $f(y) \leq f(x)$ as well, contradicting
  $f(x) < f(y)$. Whence $x < y$.
\end{proof}

\begin{lemma}
  \label{lemma:surjective}
  Any \bcont function $f : C \rTo \I$, where $C$ is a path, is
  surjective.
\end{lemma}
\begin{proof}
  Since $f$ is \bcont, it has left and right adjoints, say
  $\ell \dashv f \dashv \rho$. We shall show that $\ell \leq \rho$;
  from this and the unit/counit relations
  $h(\rho(t)) \leq t \leq h(\ell(t))$ it follows that both $\ell(t)$
  and $\rho(t)$ are preimages of $t \in \I$.
  
  Let $t \in \I$ be arbitrary; since $C$ is a chain, either
  $\ell(t) \leq \rho(t)$ holds, or $\rho(t) < \ell(t)$ holds.  In the
  latter case, let $c \in C$ be such that $\rho(t) < c < \ell(t)$. As
  $\I$ is a chain, either $f(c) \leq t$, or $t \leq f(c)$. If
  $f(c) \leq t$, then we have $c \leq \rho(t)$, contradicting
  $\rho(t) < c$; if $t \leq f(c)$, then $\ell(t) \leq c$,
  contradicting $c < \ell(t)$. Therefore the relation
  $\ell(t) \leq \rho(t)$ holds, for each $t \in C$.
\end{proof}

For a path $C \subseteq \I^{d}$ and $i=1,\ldots ,d$, let us define
$\pi_{i} : C \rto \I$ as the inclusion of $C$ into $\I^{d}$ followed
by the projection to the $i$-component.  Observe that $\pi_{i}$ is
\bcont (since it is the composition of two \bcont functions), thus it
is surjective by the previous Lemma.
\begin{proposition}
  Every path $C$ is order isomorphic to $\I$. In particular, there
  exists a monotone continuous function $\p : C \rto \I$ such that
  $\p(0) = \vec{0}$, $\p(1) = \vec{1}$, and $\p(\I) = C$.
\end{proposition}
\begin{proof}
  We shall show that $C$ has a dense countable subset $C_{\Q}$ without
  endpoints which generates $C$ both under infinite joins and under
  infinite meets. By a well known theorem by Cantor, see
  e.g. \cite[Proposition~1.4.2]{CK},
  $C_{\Q}$ is order isomorphic to
  $\IQ \setminus \set{0,1}$. Then $C$ is order isomorphic to the
  Dedekind-MacNeille completion of $\IQ \setminus \set{0,1}$, namely
  to $\I$.
  For each $i \in \set{1,\ldots ,d}$ and
  $q \in \IQ \setminus \set{0,1}$, pick $c_{i,q} \in C$ such that
  $\pi_{i}(c_{i,q}) = q$. Let
  \begin{align*}
    C_{\Q} & := \set{ c_{i,q} \mid i \in \set{1,\ldots ,d}, \;q \in
      \IQ\setminus \set{0,1}}\,,
  \end{align*}
  and observe that $C_{\Q}$ is countable. We firstly argue that
  $C_{\Q}$ is dense in $C$.  Let $c,c' \in C$ such that $c < c'$. By
  definition of the order on $\I^{d}$, $\pi_{i}(c) < \pi_{i}(c')$ for
  some $i \in \set{1,\ldots ,d}$.  Let $q \in \IQ$ be such that
  $\pi_{i}(c) < q < \pi_{i}(c')$. Then, by Lemma \ref{lemma:reflect},
  we deduce $c < c_{i,q} < c'$, with $c_{i,q} \in C_{\Q}$.

  Also $C_{\Q}$ has no endpoints. For example, if
  $c = c_{i,q} \in C_{\Q}$ and $q' \in \IQ$ is such that $q' < q$,
  then necessarily $c_{i,q'} < c_{i,q}$, so $C_{\Q}$ has no least
  element.

  Finally, we prove that $C_{\Q}$ generates $C$ under infinite joins.
  Let $c \in C$ and consider the set
  $D := \set{x \in C_{\Q} \mid x < c}$; suppose that $\bigvee D <
  c$. There exists $i \in \set{1,\ldots ,d}$ such that
  $\pi_{i}(\bigvee D) < \pi_{i}(c)$, and we can pick $q \in \Q$ such
  that $\pi_{i}(\bigvee D) < q < \pi_{i}(c)$. Let $c_{i,q}$ be such
  that $\pi_{i}(c_{i,q}) = q$, then, by Lemma~\ref{lemma:reflect}, we
  have $\bigvee D < c_{i,q} < c$. Yet, this is a contradiction, as
  $c_{i,q} \in C_{\Q}$ and $c_{i,q} < c$ imply $c_{i,q } \in D$, whence
  $c_{i,q} \leq \bigvee D$.  In a similar way, we can show that every
  element of $C$ is a meet of elements of $C_{\Q}$.
\end{proof}

\section{Paths in dimension 2}
\label{sec:pathsDimensionTwo}

We give next a further characterization of the notion of path, valid
in dimension $2$. The principal result of this Section,
Theorem~\ref{thm:pathsasjcont}, states that paths in dimension $2$
bijectively correspond to elements of the quantale $\LjI$.

\medskip

For a monotone function $f : \I \rto \I$ define
$C_{f} \subseteq \I^{2}$ by the formula
\begin{align}
  \label{def:CfDimTwo}
  C_{f} & := \bigcup_{x \in \I}\; \set{x}\times[\joinof{f}(x),\meetof{f}(x)]\,.
\end{align}
Notice that, by Proposition~\ref{prop:meet-cont-closure},
$C_{f} = C_{\joinof{f}} =C_{\meetof{f}}$. As suggested in
figure~\ref{fig:Cf}, when $f \in \LjI$, then $C_{f}$ is the graph of
$f$ (in blue in the figure) with the addition of the intervals
$(\joinof{f}(x),\meetof{f}(x)]$ (in red in the figure) when $x$ is a
discontinuity point of $f$.
\begin{figure}
  \centering
  \newcommand*{\xMin}{0}
  \newcommand*{\xMax}{6}
  \newcommand*{\yMin}{0}
  \newcommand*{\yMax}{6}
  \begin{tikzpicture}
    \foreach \i in {\xMin,...,\xMax} {
      \draw [very thin,gray] (\i,\yMin) -- (\i,\yMax)  
      ;
    }
    \foreach \i in {\yMin,...,\yMax} {
      \draw [very thin,gray] (\xMin,\i) -- (\xMax,\i) 
      ;
    }
    \lpoint{0,0} ;
    \draw [thick,red] (0,0.05) -- (0,0.5) ;
    \draw [thick,blue] (0,0.5) -- (2,1) ;
    \lpoint{2,1} ;

    \draw [thick,blue] (2,2) -- (5,3) ;  \lpoint{5,3} ;
    \draw [thick,blue] (5,4) -- (6,5) ; \lpoint{6,5} ;
    \draw [thick,red] (2,1.05) -- (2,2) ;
    \draw [thick,red] (5,3.05) -- (5,4) ;
    \draw [thick,red] (6,5.05) -- (6,6) ;

    \rpoint{0,0.5};
    \rpoint{2,2};
    \rpoint{5,4};
    \rpoint{6,6};

  \end{tikzpicture}
  \label{fig:Cf}
  \caption{The path $C_{f}$ of $f \in \LjI$}
\end{figure}

\begin{proposition}
  \label{lemma:pathd2}
  $C_{f}$ is a path in $\I^{2}$.
\end{proposition}
\begin{proof}  
  We prove first that $C_{f}$, with the product ordering induced from
  $\I^{2}$, is a linear order. To this goal, we shall argue that, for
  $(x,y),(x',y') \in C_{f}$, we have $(x,y) < (x',y')$ iff either $x <
  x'$ or $x = x'$ and $y < y'$. That is, $C_{f}$ is a lexicographic
  product of linear orders, whence a linear order.
  Let us suppose that one of these two conditions holds: a) $x < x'$,
  b) $x = x'$ and $y < y'$.  If a), then
  $\meetof{f}(x) \leq \joinof{f}(x')$.  Considering that
  $y \in [\joinof{f}(x),\meetof{f}(x)]$ and
  $y' \in [\joinof{f}(x'),\meetof{f}(x')]$ we deduce $y \leq y'$. This
  proves that $(x,y) < (x',y')$ in the product ordering.  If b) then
  we also have $(x,y) < (x',y')$ in the product ordering.
  The converse implication, $(x,y) < (x',y')$ implies $x < x'$ or $x =
  x'$ and $y < y'$, trivially holds.

  We argue next that $C_{f}$ is closed under joins from $\I^{2}$.  Let
  $(x_{i},y_{i})$ be a collection of elements in $C_{f}$, we aim to
  show that $(\bigvee x_{i}, \bigvee y_{i}) \in C_{f}$, i.e.
  $\bigvee y_{i} \in
  [\joinof{f}(\bigvee x_{i}), \meetof{f}(\bigvee x_{i})]$.
  Clearly, as $y_{i} \leq \meetof{f}(x_{i})$, then
  $\bigvee y_{i} \leq \bigvee \meetof{f}(x_{i}) \leq
  \meetof{f}(\bigvee x_{i})$. Next, $\joinof{f}(x_{i}) \leq y_{i}$,
  whence
  $\joinof{f}(\bigvee x_{i}) = \bigvee \joinof{f}(x_{i}) \leq \bigvee
  y_{i}$.
  By a dual argument, we have that $(\bigwedge x_{i}, \bigwedge y_{i})
  \in C_{f}$.
  
  Finally, we show that $C_{f}$ is dense; to this goal let
  $(x,y),(x',y') \in C_{f}$ be such that $(x,y) < (x',y')$. If
  $x < x'$ then we can find a $z$ with $x < z < x'$; of course,
  $(z,f(z)) \in C_{f}$ and, by the previous characterisation of the
  order, $(x,y) < (z,f(z)) < (x',y')$ holds.  If $x = x'$ then
  $y < y'$ and we can find a $w$ with $y < w < y'$; as
  $ w \in [y,y'] \subseteq [\joinof{f}(x),\meetof{f}(x)]$, then
  $(x,w) \in C_{f}$; clearly, we have then
  $(x,y) < (x,w) < (x,y') = (x',y')$.
\end{proof}
For $C$ a path in $\I^{2}$, define
\begin{align}
  \label{eq:defvC}
  f_{C}^{-}(x) & := \bigwedge \set{y \mid (x,y) \in C}\,,
  &
  f_{C}^{+}(x) & := \bigvee \set{y \mid (x,y) \in C}\,.
\end{align}
Recall that a path $C \subseteq \I^{2}$ comes with \bcont surjective
projections $\pi_{1},\pi_{2} : C\rto \I$. Observe that
the following relations hold:
\begin{align}
  \label{eq:exprviaadjoints}
  f^{-}_{C} & = \pi_{2} \circ \ladj{(\pi_{1})}\,, & f^{+}_{C} & =
  \pi_{2} \circ \radj{(\pi_{1})}\,.
\end{align}
Indeed, we have
\begin{align*}
  \pi_{2}(\ladj{(\pi_{1})}(x))
  & = \pi_{2}(\bigwedge \set{(x',y) \in C \mid x = x' })
   \tag*{using  equation \eqref{eq:adjsurjective}}\,,
  \\
  & = \bigwedge \pi_{2}(\set{(x',y) \in C \mid x = x' }) = \bigwedge
  \set{ y \mid (x,y) \in C }\,.
\end{align*}
The other expression for $f_{C}^{+}$ is derived similarly. In particular,
the expressions in \eqref{eq:exprviaadjoints} show that
$f_{C}^{-} \in \LjI$ and $f_{C}^{+} \in \LmI$.

\begin{lemma}
  \label{lemma:firstInverse}
  We have
  \begin{align*}
    f_{C}^{-} & = \joinof{(f_{C}^{+})}\,, \quad f_{C}^{+}  =
    \meetof{(f_{C}^{-})}\,,
    \quad \text{and} \quad C = C_{f_{C}^{+}} = C_{f_{C}^{-}}\,.
  \end{align*}
\end{lemma}
\begin{proof}
  Firstly, let us argue that $f_{C}^{+} = \meetof{(f_{C}^{-})}$; we
  do this by showing that $f_{C}^{+}$ is the least \mc function above
  $f_{C}^{-}$.  We have $f_{C}^{-}(x) \leq f_{C}^{+}(x)$ for each
  $x \in \I$, since $\pi_{1}$ is surjective so the fibers
  $\pi_{1}^{-1}(x) = \set{(x',y) \in C \mid x' = x}$ are non empty.
  Suppose now that $f_{C}^{-} \leq g \in \LmI$. In order to prove that
  $f_{C}^{+} \leq g$ it will be enough to prove that
  $f_{C}^{+}(x) \leq g(x')$ whenever $x < x'$.  Observe that if
  $x < x'$ then $f_{C}^{+}(x) \leq f_{C}^{-}(x')$: this is because if
  $(x,y),(x',y') \in C$, then $x < x'$ and $C$ a chain imply
  $y \leq y'$. We deduce therefore
  $f_{C}^{+}(x) \leq f_{C}^{-}(x') \leq g(x')$.
  The relation $f_{C}^{-} = \joinof{(f_{C}^{+})}$ is proved similarly.
  
  Next we argue that $(x,y) \in C$ if and only if
  $f_{C}^{-}(x) \leq y \leq f_{C}^{+}(x)$. The direction from left to
  right is obvious. Conversely, we claim that if
  $f_{C}^{-}(x) \leq y \leq f_{C}^{+}(y)$, then the pair $(x,y)$ is
  comparable with all the elements of $C$. It follows then that
  $(x,y) \in C$, since $C$ is a maximal chain.  Let us verify the
  claim. Let $(x',y') \in C$, if $x = x'$ then our claim is obvious,
  and if $x' < x$, then
  $y' \leq f_{C}^{+}(x') \leq f_{C}^{-}(x) \leq y$, so
  $(x',y') \leq (x,y)$; the case $x < x'$ is similar.
\end{proof}

\begin{lemma}
  \label{lemma:secondInverse}
  Let $f : \I \rto \I$ be monotone and consider the path $C_{f}$.
  Then $\joinof{f} = f_{C_{f}}^{-}$ and $\meetof{f} = f_{C_{f}}^{+}$.
\end{lemma}
\begin{proof}
  For a monotone $f : \I \rto \I$, let $f' : \I \rto C_{f}$ by
  $f' := \langle id_{\I},\joinof{f}\rangle$, so
  $\joinof{f} = \pi_{2} \circ f'$, as in the diagram below:
  \begin{center}
    \newcommand{\laone}{\langle id,\joinof{f}\rangle}
    \newcommand{\latwo}{\ladj{(\pi_{1})}}
    \begin{tikzcd}
      \I
      \ar[rrr, bend left=30, "f^{-}_{C_{f}}", shift left=0.4em]
      \ar[rrr, bend right=30, "\joinof{f}"', shift right=0.4em]
      \ar[rr,"\laone"', shift right=0.5em]
      \ar[rr, "\latwo",shift left=0.5em]
      & & C_{f} \ar[ll,"\pi_{1}" description] \ar[r,"\pi_{2}"] & \I \\
    \end{tikzcd}
  \end{center}
  Recall
  that $f^{-}_{C_{f}} = \pi_{2} \circ \ladj{(\pi_{1})}$. Therefore, in
  order to prove the relation
  $\joinof{f} = f^{-}_{C_{f}} = \pi_{2} \circ \ladj{(\pi_{1})}$ it
  shall be enough to prove that $\langle id, \joinof{f}\rangle$ is
  left adjoint to the first projection (that is, we prove that
  $\langle id, \joinof{f}\rangle = \ladj{(\pi_{1})}$, from which it
  follows that
  $\joinof{f} = \pi_{1}\circ \langle id, \joinof{f}\rangle = \pi_{2}
  \circ \ladj{(\pi_{1})}$).  This amounts to verify that, for
  $x \in \I$ and $(x',y) \in C_{f}$ we have $x \leq \pi_{1}(x',y)$ if
  and only if $(x,\joinof{f}(x)) \leq (x',y)$. To achieve this goal,
  the only non trivial observation is that if $x \leq x'$, then
  $\joinof{f}(x) \leq \joinof{f}(x') \leq y$.
  The relation $\meetof{f} =
  \pi_{2} \circ \radj{(\pi_{1})}$ is proved similarly.
\end{proof}

\begin{theorem}
  \label{thm:pathsasjcont}
  There is a bijective correspondence between the following data:
  \begin{enumerate}
  \item paths in $\I^{2}$,
  \item \jc functions in $\LjI$,
  \item \mc functions in $\LmI$.
  \end{enumerate}
\end{theorem}
\begin{proof}
  According to Lemmas~\ref{lemma:firstInverse} and
  \ref{lemma:secondInverse}, the correspondence sending a path $C$ to
  $f_{C}^{-} \in \LjI$ has the mapping sending $f$ to $C_{f}$ as an
  inverse. Similarly, the correspondence
  $C \mapsto f_{C}^{+} \in \LmI$ has $f \mapsto C_{f}$ as
  inverse.
\end{proof}

\section{Paths in higher dimensions}
\label{sec:pathsDimensionMore}

We show in this section that paths in dimension $d$, as defined in
Section~\ref{sec:paths}, are in bijective correspondence with
clopen tuples of $\PrLI$, as defined in
Section~\ref{sec:latticesFromQuantales}; therefore, as established in
that Section, there is a lattice $\Ld{\LjI}$ whose underlying set can
be identified with the set of paths in dimension $d$.

\medskip

Let $f \in \PrLI$, so $f = \set{ f_{i,j} \mid 1 \leq i < j \leq
  d}$. We define then, for $1 \leq i < j \leq d$,
\begin{align*}
  f_{j,i} & := \opp{(\,f_{i,j}\,)} = \joinof{(\radj{(f_{i,j})})} \,.
\end{align*}
Moreover, for $i \in [d]$, we let $f_{i,i} := id$.  We say shall say
that a tuple $f \in \PrLI$ is \emph{compatible} if
$f_{j,k} \circ f_{i,j} \leq f_{i,k}$,
for each triple of elements $i,j,k \in [d]$. It is readily seen that a
tuple is compatible if and only if $\delta_{f} $, defined in
Lemma~\ref{lemma:defskewmetric}, is a skew metric on
$\setof{d}$. Therefore, according to Lemma~\ref{lemma:enrichment},
\emph{a tuple is compatible if and only if it is clopen.}

If $C \subseteq \I^{d}$ is a path, then we shall use
$\pi_{i} : C \rto \I$ to denote the projection onto the $i$-th
coordinate. Then
$\pi_{i,j} := \langle \pi_{i},\pi_{j}\rangle : C \rto \I \times \I$.
\begin{definition}
  \label{def:defvCij}
  For a path $C$ in $\I^{d}$, let us define $v(C) \in \PrLI[d]$ by the
  formula:
  \begin{align}
    \label{eq:defvCij}
    v(C)_{i,j} & := \pi_{j}\circ \ladj{(\pi_{i})}\,, \quad (i,j) \in \couples{d}. 
  \end{align}
\end{definition}
\begin{remark}
  An explicit formula for $v(C)_{i,j}(x)$ is as follows:
  \begin{align}
    \label{eq:explicitvC}
    v(C)_{i,j}(x)
    & = \bigwedge \set{\pi_{j}(y) \in C\mid \pi_{i}(y) = x}\,.
  \end{align}
  Let $C_{i,j}$ be the image of $C$ via the projection
  $\pi_{i,j}$. Then $C_{i,j}$ is a path, since it is the image of a
  bi-continuous function from $\I$ to $ \I \times \I$.  Some simple
  diagram chasing (or the formula in~\eqref{eq:explicitvC}) shows that
  $v(C)_{i,j} = f^{-}_{C_{i,j}}$ as defined in \eqref{eq:defvC}.
\end{remark}

\begin{definition}
  For a compatible $f \in
  \PrLI$, 
  define
  \begin{align*}
    C_{f} & := \set{(x_{1},\ldots ,x_{d}) \mid f_{i,j}(x_{i}) \leq
      x_{j}, \text{ for all } i,j \in [d] }\,.
  \end{align*}
\end{definition}
\begin{remark}
  \label{remark:adj}
  Notice that the condition $f_{i,j}(x) \leq y$ is equivalent (by
  definition of $f_{i,j}$ or $f_{j,i}$) to the condition
  $x \leq \Meetof{f_{j,i}}(y)$.  Thus, there are in principle many
  different ways to define $C_{f}$; in particular, when $d = 2$ (so
  any tuple $\PrLI$ is compatible), the definition given above is
  equivalent to the one given in \eqref{def:CfDimTwo}.
\end{remark}

\begin{proposition}
  $C_{f}$ is a path.
\end{proposition}
The proposition is an immediate consequence of the following
Lemmas~\ref{lemma:Cftotal}, \ref{lemma:Cfcomplete} and
\ref{lemma:Cfdense}.

\begin{lemma}
  \label{lemma:Cftotal}
  $C_{f}$ is a total order.
\end{lemma}
\begin{proof}
  Let $x,y \in C_{f}$ and suppose that $x\not\leq y$, so there exists
  $i \in [d]$ such that $x_{i} \not\leq y_{i}$. W.l.o.g. we can
  suppose that $i = 1$, so $y_{1} < x_{1}$ and then, for $i > 1$, we have 
  $\Meetof{f_{1,i}}(y_{1}) \leq f_{1,i}(x_{1})$, whence $y_{i} \leq
  \Meetof{f_{1,i}}(y_{1}) \leq f_{1,i}(x_{1}) \leq x_{1}$.
  This shows that $y < x$.
\end{proof}

\begin{lemma}
  \label{lemma:Cfcomplete}
  $C_{f}$ is closed under arbitrary meets and joins.
\end{lemma}
\begin{proof}
  Let $\set{x^{\ell} \mid \ell \in I}$ be a family of tuples in
  $C_{f}$.  For all $i,j \in \setof{d}$ and $\ell \in I$, we have
  $f_{i,j}(\bigwedge_{\ell \in I} x^{\ell}_{i})
   \leq f_{i,j}(x^{\ell}_{i}) \leq
   x^{\ell}_{j} $,
  and therefore
  $f_{i,j}(\bigwedge_{\ell \in I} x^{\ell}_{i}) \leq \bigwedge_{\ell
    \in I}x^{\ell}_{j}$. Since meets in $\I^{d}$ are computed
  coordinate-wise, this shows that $C_{f}$ is closed under arbitrary
  meets.  Similarly,
  $ f_{i,j}(x^{\ell}_{i}) \leq \bigvee_{\ell \in I}x^{\ell}_{j}$ and
  \begin{align*}
    f_{i,j}(\bigvee_{\ell \in I} x^{\ell}_{i}) & = \bigvee_{\ell \in
      I} f_{i,j}(x^{\ell}_{i}) \leq \bigvee_{\ell \in I}x^{\ell}_{j}
    \,,
  \end{align*}
  so $C_{f}$ is also closed under arbitrary joins.
\end{proof}
\begin{lemma}
  \label{lemma:Crightadj}
  Let $f \in \PrLI$ be compatible.  Let $i_{0} \in \setof{d}$ and
  $x_{0} \in \I$; define $x \in \I^{d}$ by setting
  $x_{i} := f_{i_{0},i}(x_{0})$ for each $i \in \setof{d}$.
  Then $x \in C_{f}$ and $x = \bigwedge \set{y \in C_{f} \mid
    \pi_{i_{0}}(y) = x_{0} }$.
\end{lemma}
\begin{proof}
  Since $f$ is compatible, $ f_{i,j} \circ f_{i_{0},i} \leq
  f_{i_{0},j}$, for each $i, j \in\setof{d}$, so
  \begin{align*}
     f_{i,j}(x_{i}) & = f_{i,j}(f_{i_{0},i}(x_{0})) \leq 
     f_{i_{0},j}(x_{0})  = x_{j}\,.
  \end{align*}
  Therefore, $x \in C_{f}$. Observe that since $f_{i_{0},i_{0}} =
  id$, we have $x_{i_{0}} = x_{0}$ and
  $x$ so defined is such that $\pi_{i_{0}}(x) = x_{0}$.
  On the other hand, if $y \in C_{f}$ and $x_{0} \leq \pi_{i_{0}}(y)
  =y_{i_{0}}$, then $x_{i} = f_{i,i_{0}}(x_{0}) \leq
  f_{i,i_{0}}(y_{i_{0}}) \leq y_{i}$, for all $i
  \in\setof{d}$.  Thus $x = \bigwedge \set{y \in C_{f} \mid
    \pi_{i_{0}}(y) = x_{0} }$. 
\end{proof}
\begin{lemma}
  \label{lemma:Cfdense}
  $C_{f}$ is dense.
\end{lemma}
\begin{proof}
  Let $x,y \in C_{f}$ and suppose that $x < y$, so there exists $i_{0}
  \in \setof{d}$ such that $x_{i_{0}}< y_{i_{0}}$. Pick $z_{0} \in
  \I$ such that $x_{i_{0}} < z_{0} < y_{i_{0}}$ and define
  $z \in  C_{f}$ as in  Lemma~\ref{lemma:Crightadj}, 
  $z_{i} := f_{i_{0},i}(z_{0})$, for
  all $i \in \setof{d}$.
  We claim that $x_{i} \leq z_{i} \leq y_{i}$, for each $i \in
  \setof{d}$. From this and $x_{i_{0}} < z_{i_{0}} <
  y_{0}$ it follows that $x < z < y$. Indeed, we have
  $z_{i}  = f_{i_{0},i}(z_{0}) \leq f_{i_{0},i}(y_{i_{0}}) \leq
    y_{i}$.
    Moreover, $x_{i_{0}} <
    z_{0}$ implies $\Meetof{f_{i_{0},i}}(x_{i_{0}}) \leq
    f_{i_{0},i}(z_{0})$; by Remark~\ref{remark:adj}, we have $x_{i }
    \leq
    \Meetof{f_{i_{0},i}}(x_{i_{0}})$. Therefore, we also have $x_{i}
    \leq \Meetof{f_{i_{0},i}}(x_{i_{0}}) \leq f_{i_{0},i}(z_{0}) =
    z_{i}$.
\end{proof}

\begin{lemma}
  \label{lemma:vCfeqf}
  If $f \in \PrLI$ is compatible, then $v(C_{f}) = f$.
\end{lemma}
\begin{proof}
  By Lemma~\ref{lemma:Crightadj}, the correspondence sending
  $x$ to $(f_{i,1}(x),\ldots
  ,f_{d,1}(x))$ is left adjoint to the projection $\pi_{i} : C_{f}
  \rto
  \I$. 
  In turn, this gives that $v(C_{f})_{i,j}(x) =
  \pi_{j}(\ladj{(\pi_{i})}(x)) = f_{i,j}(x)$, for any $i,j \in
  \setof{d}$. It follows that $v(C_{f}) = f$.  
\end{proof}
\begin{lemma}
  \label{lemma:CvCeqC}
  For $C$ a path in $\I^{d}$,  we have $C_{v(C)} = C$.
\end{lemma}
\begin{proof}
  Let us show that $C \subseteq C_{v(C)}$. Let $c \in C$;  notice that
  for each $i, j \in\setof{d}$, we have 
  \begin{align*}
    v(C)_{i,j}(c_{i}) & = \pi_{j}(\ladj{(\pi_{i})}(c_{i}))
    = \pi_{j}(\ladj{(\pi_{i})}(\pi_{i}(c)) \leq \pi_{j}(c) = c_{j}\,,
  \end{align*}
  so $c \in C_{v(C)}$. For the converse inclusion, notice that
  $C \subseteq C_{v(C)}$ implies $C = C_{v(C)}$,  since every path is
  a maximal chain.
\end{proof}

Putting together Lemmas~\ref{lemma:vCfeqf} and \ref{lemma:CvCeqC} we
obtain:
\begin{theorem}
  The correspondences, sending a path $C$ in $\I^{d}$ to the tuple
  $v(C)$, and a compatible tuple $f$ to the path $C_{f}$, are inverse
  bijections.
\end{theorem}

\section{Structure of the \cwo{s}}
\label{sec:weakBruhat}

As established in Section~\ref{sec:latticesFromQuantales}, there is a
lattice structure $\Ld{\LjI}$ whose underlying set is the set of
clopen tuples of the product $\PrLI$. By the results in the previous
section, these tuples can be identified with paths in dimension $d$.
We give in this section a minimum of structural theory of these 
lattices by characterizing their \jirr elements.

\subsection{Join-prime elements of $\LjI$}
Recall from Corollary~\ref{cor:distrlattice} that $\LjI$ is a complete
distributive lattice and that, in distributive lattices, join-prime
and join-irreducible elements coincide.  We determine therefore the
join-prime elements of $\LjI$.  For $x,y \in \I$, let us put
\begin{align}
  \ji{x,y}(t)
  & :=
  \begin{cases}
    \,0\,, & 0\leq t \leq x\,, \\
    \,y\,, &  x < t \,,
  \end{cases}
  &
  \JI{x,y}(t)
  & :=
  \begin{cases}
    \,0\,, & 0\leq t < x\,, \\
    \,y\,, & x \leq t < 1 \,, \\
    \,1\,, & t = 1\,,
  \end{cases}
\end{align}
so $\ji{x,y} \in \LjI$, $\JI{x,y} \in \LmI$ and $\JI{x,y} = \Ji{x,y}$. 
\begin{definition}
   A \emph{\osf} is a function of the form $\ji{x,y}$ 
   where $x,y \in \I$.
   We say that $\ji{x,y}$ is \emph{prime} if 
   $\ji{x,y}\neq \bot$.
   We say that $\ji{x,y}$ is \emph{rational} if $x,y \in \IQ$.
 \end{definition}
 \begin{lemma}
   For each $x,y \in \I$, $\ji{x,y} = \bot$ if and only of $x = 1$ or
   $y = 0$.
 \end{lemma}
 \begin{proof}
   If $x = 1$ or $y = 0$, then $\ji{x,y}$ is the constant
   function that takes $0$ as its unique value, \ie $\ji{x,y} = \bot$.
   Conversely, if $x < 1$ and $0 < y$, then,
   $\ji{x,y}(1) = y \neq 0$, so $\ji{x,y} \neq \bot$.
 \end{proof}
 From the lemma it also follows that $\ji{x,y} \neq \bot$ if and only
 if $x < 1$ and $0 < y$. Notice therefore that $\ji{x,y} \neq \bot$ if
 and only if the point $(x,y) \in \I^{2}$ does not lie on the path
 $\set{(x,0) \mid x \in \I} \cup \set{(1,y) \mid y \in \I}$.

 \begin{lemma}
   \label{lemma:lessjif}
  For $f \in \LjI$ and $x,y \in \I$, $\ji{x,y} \leq f$ if and only if
  $y \leq \MeetOf{f}(x)$.
\end{lemma}
\begin{proof}
  If $\ji{x,y} \leq f$ then
  $y = \MeetOf{\ji{x,y}}(x) \leq \MeetOf{f}(x)$. Conversely, suppose
  that $y \leq \MeetOf{f}(x)$. If $t \leq x$, then
  $\ji{x,y}(t) = 0 \leq f(t)$. If $x < t \leq 1$, then
  $\ji{x,y}(t) = y \leq \MeetOf{f}(x) \leq f(t)$, where the last
  inequality follows from Lemma~\ref{lemma:meetofbelowjoinof}.
\end{proof}
\begin{corollary}
  \label{cor:orderonjp}
  Let $x,y,z,w \in \I$ and suppose that $\ji{x,y},\ji{z,w}\neq \bot$.
  Then $\ji{x,y} \leq \ji{z,w}$ if and
  only if $z \leq x$ and $y \leq w$.
\end{corollary}
\begin{proof}
  If $\ji{x,y} \leq \ji{z,w}$, then $y \leq \Ji{z,w}(x)$. Since
  $0 < y $, we derive then $z \leq x$. Since $x < 1$ we also have
  $\Ji{z,w}(x) = w$, so $y \leq \Ji{z,w}(x) = w$.
  Conversely, suppose $z \leq x$ and $y \leq w$. From $z \leq x < 1$
  we deduce $\Ji{z,w}(x) = w$, so $y \leq w = \Ji{z,w}(x)$ yields,
  according to the previous lemma, $\ji{x,y} \leq \ji{z,w}$.
\end{proof}

For $f \in \LjI$ and $x_{0},x_{1} \in \I$ with $x_{0} \leq x_{1}$, we
define $f_{(x_{0},x_{1}]} \in \LjI$ as follows:
\begin{align*}
  f_{(x_{0},x_{1}]}(t) & :=
  \begin{cases}
    0 \,, & 0 \leq t \leq x_{0}\,, \\
    f(t)\,, & x_{0} < t \leq x_{1}\,, \\
    \MeetOf{f}(x_{1})\,, & x_{1} < t \,.
  \end{cases}
\end{align*}
In particular, for any $x \in \I$, we have
\begin{align*}
  f_{(0,x]}(t)
  & = 
  \begin{cases}
    f(t)\,, & 0 \leq t \leq x\,, \\
    \MeetOf{f}(x)\,, & x < t \,,
  \end{cases}
  &
  f_{(x,1]}(t)
  & = 
  \begin{cases}
    0 \,, & 0 \leq t \leq x\,, \\
    f(t)\,, & x < t \leq 1 \,,
  \end{cases}
\end{align*}
so
\begin{align*}
  f & = f_{(0,x]} \vee f_{(x,1]}\,.
\end{align*}
\begin{proposition}
  \label{prop:join-prime}
  Prime \osf{s} are exactly the join-prime elements of
  $\LjI$.
\end{proposition}
\begin{proof}
  Consider $\ji{x,y}$ and suppose that $\ji{x,y} \leq f \vee g$. This
  relation holds if and only if
  $y \leq \max (\MeetOf{f}(x),\MeetOf{g}(x))$, if and only if
  $y \leq \MeetOf{f}(x)$ or $y \leq \MeetOf{g}(x)$, that is
  $\ji{x,y} \leq f$ or $\ji{x,y} \leq g$. Thus every function of the
  form $\ji{x,y}$ which is different from $\bot$ is \jp.
  
  Conversely, let $f \in \LjI$ be \jp (so $f$ is \jirr) and recall
  that, for any $x \in \I$, $f = f_{(0,x]} \vee f_{(x,1]}$. Therefore,
  for each $x\in \I$, $f = f_{(0,x]}$ or $f = f_{(x,1]}$.
  Observe also
  that if $f = f_{(0,x]}$ and $f = f_{(x,1]}$, then
  $f = \ji{x,\MeetOf{f}(x)}$.

  Let now $I_{f} := \set{x \in \I \mid f = f_{(x,1]}}$ and
  $F_{f} := \set{x \in \I \mid f = f_{(0,x]}}$, so
  $I_{f} \cup F_{f} = \I$.  Notice that $x \in I_{f}$ if and only if
  $f(x) = 0$ and $x \in F_{f}$ if and only if the restriction of $f$
  to the interval $(x,1]$ 
  is constant.  From these considerations it immediately follows that
  $I_{f}$ is a downset and $F_{f}$ is an upset; moreover, $I_{f}$ is
  closed under joins (since $f$ is \jcont) and $F_{f}$ is closed under
  meets.  If $x \in I_{f}$, $y \in F_{f}$, and $y < x$, then $f$ is
  constant with value $0$, which contradicts $f$ being \jirr (thus
  distinct from $\bot$).  Therefore, if $x \in I_{f}$ and
  $y \in F_{f}$, then $x \leq y$.  Then
  $x_{0} = \bigvee I_{f} = \bigwedge F_{f} \in I_{f} \cap F_{f}$ and
  $f = \ji{x_{0},\MeetOf{f}(x_{0})}$.
\end{proof}

\begin{proposition}
  \label{prop:generationDim2}
  Every $f \in \LjI$ is a (possibly infinite) join of prime \osf{s}.
\end{proposition}
\begin{proof}
  Clearly we have
  $\bigvee \set{\ji{x,y} \mid \ji{x,y} \leq f} \leq f$, so let us
  argue that this inclusion is an equality. Let $g$ be such that
  $\ji{x,y} \leq g$ whenever $\ji{x,y} \leq f$. In particular, for $x$
  arbitrary and $y = \MeetOf{f}(x)$, we have $\ji{x,y} \leq g$, that
  is $\MeetOf{f}(x) \leq \MeetOf{g}(x)$. We argued therefore that,
  within $\LmI$, $\MeetOf{f} \leq \MeetOf{g}$. We have, therefore,
  $f = \JOINOF{\MeetOf{f}} \leq \JOINOF{\MeetOf{g}} = g$.
\end{proof}

\begin{remark}
  \label{remark:QjIcompletion}
  Proposition~\ref{prop:generationDim2} implies that $\QjI$ is the
  Dedekind-MacNeille completion of the sublattice generated by the
  prime \osf{s}.  The statement of the Proposition can be further
  strengthened as follows: every $f \in \QjI$ is a (possibly infinite)
  join of prime \emph{rational} \osf{s}, implying that $\QjI$ is the
  Dedekind-MacNeille completion of the sublattice generated by the
  rational \osf{s}. To see why this is the case, observe that every
  \osf is the the join of the rational \osf{s} below it.
\end{remark}

\medskip

Finally, we verify the following relations, that we shall need to
understand the structure of \jirr elements in higher dimensions.
\begin{lemma}
  \label{lemma:compJI}
  For each $x,y,y',z \in \I$,
  \begin{align*}
    \ji{y',z} \circ \ji{x,y}
    & =
    \begin{cases}
      \bot\,, & y \leq y'\,, \\
      \ji{x,z}\,, & \ttoth\,.
    \end{cases}
  \end{align*}
  In
  particular, $\ji{y,z} \circ \ji{x,y} = \bot$. 
\end{lemma}
\begin{proof}
  Let us study the formula for the composition:
  \begin{align*}
    \ji{y',z}(\ji{x,y}(t))
    & =
    \begin{cases}
      0\,, & \ji{x,y}(t) \leq y'\,, \\
      z\,, & y' < \ji{x,y}(t)\,. 
    \end{cases}
  \end{align*}
  Now, if $y \leq y'$, then $\ji{x,y}(t) \leq y'$, for each
  $t \in \I$, so $\ji{y',z}\circ \ji{x,y} = \bot$.  If
  $y' < y$, then $y' < \ji{x,y}(t)$ if and only if
  $\ji{x,y}(t) = y$, i.e. iff $x < t$. This yields
  $\ji{y',z}\circ \ji{x,y} = \ji{x,z}$.  
\end{proof}
The following Lemma is verified in a similar way.
\begin{lemma}
  \label{lemma:compJIbis}
  For each $x,y,z \in \I$, 
  \begin{align*}
    \Ji{y,z} \circ \Ji{x,y} & =
    \begin{cases}
      \Ji{y,z}\,, & y = 0\,, \\
      \Ji{x,z} \,, & 0 < y < 1\,,\\
      \Ji{x,y}\,, & y = 1\,.
    \end{cases}
  \end{align*}
\end{lemma}

\subsection{Join-irreducible elements of $\LId$}
\label{sec:jirrs}
We study next \jirr elements of the lattice $\Ld{\LjI}$, for
$d \geq 3$. To ease reading, we shall use the notation $\LId$ for
$\Ld{\LjI}$.

\medskip

For $p \in \I^{d}$, let $\ji{p} \in \PrLI$ be the tuple defined as
follows:
\begin{align*}
  \ji{p} & := \Fam{\ji{p_{i},p_{j}} \mid (i,j) \in \cd}\,.
\end{align*}
Let also define 
\begin{align*}
  \mji{p} &:= \min \set{i \in [d] \mid p_{i} < 1}\,,
  & \Mji{p} := \max \set{j \in [d] \mid 0 < p_{j}}\,,
\intertext{(where we let in  these formulas $\min \emptyset = d+1$ and
$\max \emptyset = 0$) and}
  \dimji(p) & := \Mji{p} - \mji{p}\,. 
\end{align*}
Therefore, for each $p \in \I^{d}$, $p_{i} =1$ if $i < \mji{p}$ and
$p_{j} =0$ if $j > \Mji{p}$. In particular, we cannot have
$\Mji{p} < \mji{p} -1$, so $\dimji(p) \geq -1$.

\begin{lemma}
  For each $p \in \I^{d}$, the relation $\ji{p} \neq \bot$ holds
  if and only if $\dimji(p) > 0$.
\end{lemma}
\begin{proof}
  Recall that $\ji{p_{i},p_{j}} = \bot$, if $p_{i} = 1$ or
  $p_{j} = 0$.
  Suppose that $\dimji(p) \leq 0$, so $\Mji{p} \leq \mji{p}$ and
  consider $(i,j) \in \cd$: we have then $p_{i} = 1$ or $p_{j} =
  0$. Therefore, $\ji{p_{i},p_{j}} = \bot$ for each $(i,j) \in \cd$,
  and $\ji{p} = \bot$.

  Suppose next that $\dimji(p) > 0$, so $\mji{p} < \Mji{p}$.  To
  ease reading, let $\mu = \mji{p}$ and $M = \Mji{p} $. Since
  $1 \leq \mu$ and $M \leq d$, we have $(\mu,M) \in \cd$ and since
  $p_{\mu} \neq 1$ and $p_{M} \neq 0$, we have
  $\ji{p_{\mu},p_{M}} \neq \bot$ and therefore $\ji{p} \neq \bot$.
\end{proof}

\begin{proposition}
  \label{prop:clopen}
  For each $p \in \I^{d}$, $\ji{p} $ is a clopen tuple of
  $\PrLI$. That is, $\ji{p} \in \LId$.
\end{proposition}
\begin{proof}
  We use Remark~\ref{rem:suffCondClopen} to establish that $\ji{p}$ is
  compatible and, to this goal, we use the relations established with
  Lemmas~\ref{lemma:compJI} and \ref{lemma:compJIbis}. The relation
  $\Ji{p_{i},p_{k}} = \Ji{p_{j},p_{k}}\circ \Ji{p_{i},p_{j}}$ holds
  unless $p_{j} \in \set{0,1}$. If $p_{j } = 0$, then
  \begin{align*}
    \ji{p_{j},p_{k}} \circ \ji{p_{i},p_{j}} & = \bot \leq
    \ji{p_{i},p_{k}} \leq \Ji{p_{i},p_{k}} \leq \Ji{0,p_{k}} =
    \Ji{p_{j},p_{k}} \circ \Ji{p_{i},p_{j}}\,.
  \end{align*}
  If $p_{j} =1$, then
  \begin{align*}
    \ji{p_{j},p_{k}} \circ \ji{p_{i},p_{j}} & = \bot \leq
    \ji{p_{i},p_{k}} \leq \Ji{p_{i},p_{k}} \leq
    \Ji{p_{i},1} = \Ji{p_{j},p_{k}} \circ
    \Ji{p_{i},p_{j}}\,.
    \tag*{\qedhere}
  \end{align*}
\end{proof}

Notice that $p \in \I^{d}$ has $\dimji(p) \leq 0$ if and only if it
lies on the path
\begin{align*}
  \bigcup_{i \in \setof{d}} \set{(\underbrace{1,\ldots
      ,1}_{i-1},x,\underbrace{0,\ldots,0}_{d -i}) \mid x \in
  \I}\,.
\end{align*}
It is readily seen that this path corresponds to the tuple that is the
bottom of the lattice $\LId$ (as well as of the lattice $\PrLI$).

\begin{lemma}
  \label{lemma:approx}
  For each $f \in \LId$, $x \in \I$, and $(m,M) \in \cd$,
  there exists $p(f,x,m,M) \in \I^{d}$ such that $\ji{p(f,x,m,M)} \leq
  f$, $p(f,x,m,M)_{m} = x$ and $p(f,x,m,M)_{M} = \MeetOf{f_{m,M}}(x)$.
\end{lemma}
\begin{proof}
  We construct $p = p(f,x,m,M)$ as follows.  We let $p_{m} = x$ and,
  for $i$ with $m < i \leq M$, we let $p_{i} = \MeetOf{f_{m,i}}(x)$.  If $i <
  m$ then we let $p_{i} = 1$, and if $M < i$, then we let $p_{i} = 0$.
  
  Let us verify that $\ji{p} \leq f$, that is
  $\ji{p_{i},p_{j}} \leq f_{i,j}$ for each $(i,j) \in \cd$.
  If $i < m$, then $\ji{p_{i},p_{j}} = \bot \leq f_{i.j}$.
  Similarly, if $M < j$, then $\ji{p_{i},p_{j}} = \bot \leq f_{i,j}$.
  Therefore we can assume that $m \leq i < j \leq M$.  We verify that
  $e_{p_{i},p_{j}} \leq f_{i,j}$ using Lemma~\ref{lemma:lessjif}. If
  $i = m$, $p_{j} = \MeetOf{f_{m,j}}(x) \leq \MeetOf{f_{m,j}}(p_{m})$,
  simply because $p_{m} = x$; if $m < i$, then
  $p_{j} = \MeetOf{f_{m,j}}(x) \leq
  \MeetOf{f_{i,j}}(\MeetOf{f_{m,i}}(x)) = \MeetOf{f_{i,j}}(p_{i})$,
  recalling that $p_{i} = \MeetOf{f_{m,i}}(x)$ and using openedness of
  $f$.
\end{proof}

\begin{corollary}
\label{cor:joinofeps}
For each $f \in \LI[d]$ the relation
  \begin{align}
    \label{equ:joinofjps}
    f & = \bigvee \set{\ji{p} \mid p \in \I^{d} \;\;\tand\;\; \ji{p} \leq f}\,.
  \end{align}
  holds in $\LI[d]$.
\end{corollary}
\begin{proof}
  Using Lemma \ref{lemma:approx}, we see that relation
  \eqref{equ:joinofjps} holds in $\PrLI$, for any $f \in \LId$. A
  fortiori, the same relation holds in $\LId$.  
\end{proof}

\begin{proposition}
  \label{prop:jipji}
  For each $p \in \I^{d}$, if $\ji{p} \neq \bot$, then $\ji{p}$ is
  \jirr within $\LId$.
\end{proposition}
\begin{proof}

  Assume that the relation $\ji{p} = \alpha \vee \beta$ holds in
  $\LId$.  Let $m := \mji{p}$ and $M := \Mji{p}$.  Observe that, for
  $(i,j) \in \cd$, if $i < m$ or $M < j$, then
  $\ji{i,j} = \alpha_{i,j} = \beta_{i,j} = \bot$; so we only need to
  show that either $\ji{p_{i,j}} = \alpha_{i,j}$ whenever
  $m \leq i < j \leq M$, or $\ji{p_{i,j}} = \beta_{i,j}$ whenever
  $m \leq i < j \leq M$. Said otherwise, we can assume that $m = 1$
  and $M = d$ (so $p_{1} \neq 1$ and $p_{d} \neq 0$).

  Firstly, we claim that
  $\ji{p_{1},p_{d}} = \alpha_{1,d} \vee \beta_{1,d}$. If not, then we
  have $\alpha_{1,d} \vee \beta_{1,d} < \ji{p_{1},p_{d}}$; consider
  then the tuple $f \in \PrLI$ such that $f_{i,j} = \ji{p_{i},p_{j}}$
  if $i \neq 1$ or $j \neq d$, and
  $f_{1,d} = \alpha_{1,d} \vee \beta_{1,d}$; trivially, $f$ is closed,
  since for $i < j < k$,
  $\ji{p_{j},p_{k}} \circ \ji{p_{i},p_{j}} = \bot \leq
  \ji{p_{i},p_{k}}$.
  We obtain then the following contradiction:
  \begin{align*}
    \ji{p} & = \alpha \vee_{\LI[d]} \beta
    = \closure{\alpha \vee_{\PrLI} \beta}
    \leq \closure{f} = f < \ji{p}\,.
  \end{align*}
  Thus we have $\ji{p_{1},p_{d}} = \alpha_{1,d} \vee \beta_{1,d}$ in
  $\LjI$, and therefore $\ji{p_{1},p_{d}} = \alpha_{1,d}$ or
  $\ji{p_{1},p_{d}} = \beta_{1,d}$. Let us suppose that
  $\ji{p_{1},p_{d}} = \alpha_{1,d}$, we shall prove that
  $\ji{p} = \alpha$. (A similar argument proves that
  $\ji{p} = \beta$ if $\ji{p_{1},p_{d}} = \beta_{1,d}$).
  
  Notice first that $\ji{p} = \alpha \vee \beta$ implies
  $\alpha_{i,j} \leq \ji{p_{i},p_{j}}$ for each $(i,j) \in \cd$.  On
  the other hand, if $1 < i < d$, then
  \begin{align*}
    p_{d} & = \Ji{p_{1},p_{d}}(p_{1}) = \MeetOf{\alpha_{1,d}}(p_{1}) \leq
    \MeetOf{\alpha_{i,d}}(\MeetOf{\alpha_{1,i}}(p_{1})) 
    \leq \MeetOf{\alpha_{i,d}}(\Ji{p_{1},p_{i}}(p_{1})) = \MeetOf{\alpha_{i,d}}(p_{i})\,,
  \end{align*}
  showing that $\ji{p_{i},p_{d}} \leq \alpha_{i,d}$ and, consequently,
  $\ji{p_{i},p_{d}} = \alpha_{i,d}$, for $i = 1,\ldots ,d-1$.

  Suppose now that $\ji{p_{i},p_{j}} \not\leq \alpha_{i,j}$ for some
  $(i,j) \in \cd$ with $j < d$. This relation  amounts to
  $p_{j} \not\leq \MeetOf{\alpha_{i,j}}(p_{i})$, thus to $\MeetOf{\alpha_{i,j}}(p_{i}) <
  p_{j}$. Then
  \begin{align*}
    p_{d} = \Ji{p_{i},p_{d}}(p_{i}) = \MeetOf{\alpha_{i,d}}(p_{i})
    \leq  \MeetOf{\alpha_{j,d}}( \MeetOf{\alpha_{i,j}}(p_{i}))
    = \Ji{p_{j},p_{d}}(\MeetOf{\alpha_{i,j}}(p_{i}))
    = 0\,,
  \end{align*}
  against the hypothesis. Thus we have
  $\alpha_{i,j} = \ji{p_{i},p_{j}}$ for each $(i,j) \in \cd$, that is
  $\ji{p} = \alpha$.
\end{proof}

In the rest of this section we aim shall prove the converse of
Proposition \ref{prop:jipji}: if $\alpha$ is join-irreducible, then
$\alpha = \ji{p}$ for some $p \in \I^{d}$.  Let $p,q \in \I^{d}$ with
$p \leq q$; as usual, $[p,q]$ denotes the set
$\set{r \in \I^{d} \mid p \leq r \leq q}$; define then $\ji{[p,q]}$ by
\begin{align*}
  \ji{[p,q]} & := \bigvee \set{\ji{r} \mid r \in [p,q]}
\end{align*}
where this infinite join is computed in $\PrLI$ (we shall argue few
lines below that $\ji{[p,q]}$ is clopen). We notice in the meantime
the following expression of $\ji{[p,q]}$. For each $(i,j) \in \cd$,
\begin{align*}
    (\,\ji{[p,q]}\,)_{i,j} & = (\,\bigvee \set{\ji{r} \mid r \in
      [p,q]}\,)_{i,j}
    \\
    &
    = \bigvee \set{ \ji{r_{i},r_{j}}
        \mid p_{i} \leq r_{i} \leq q_{i}\,,
        \;p_{j} \leq r_{j} \leq q_{j}
      }  = \ji{p_{i},q_{j}}\,.
\end{align*}
\begin{lemma}
  \label{lemma:cube}
  For each $p,q \in \I^{d}$ with $p \leq q$,
  $\ji{[p,q]}$ is clopen. Moreover, for $r \in \I^{d}$ such that
  $\bot < \ji{r}$, the relation $\ji{r} \leq \ji{[p,q]}$ holds if and
  only
  \begin{enumerate}
  \item $p_{\mji{r}} \leq r_{\mji{r}}$ and
    $r_{\Mji{r}} \leq q_{\Mji{r}}$, 
  \item $r_{i} \in [p_{i},q_{i}]$, for each $i$ such that
    $\mji{r} < i < \Mji{r}$.
  \end{enumerate}
\end{lemma}
\begin{proof}
  Clearly $\ji{[p,q]}$ is open since it is a join of open tuples;
  it is also closed since the relations $\ji{p_{j},q_{k}} \circ \ji{p_{i},q_{j}}
  \leq \ji{p_{i},q_{k}}$
  holds by Lemma~\ref{lemma:compJI}.

  Let now $r \in \I^{d}$ be such that $\bot < \ji{r}$; to ease the
  reading, let also $m:= \mji{r}$ and $M := \Mji{r}$. Suppose that
  $\ji{r}\leq \ji{[p,q]}$, that is $ \ji{r_{i},r_{j}} \leq \ji{p_{i},q_{j}}$,
  for each $(i,j) \in \cd$.  Since
  $\bot < \ji{r_{m},r_{M}} \leq \ji{p_{m},q_{M}}$, we have
  $p_{m}\leq r_{m}$ and $r_{M} \leq q_{M}$.
  Also, for $m < i < M$, $\ji{r_{m},r_{i}} \leq \ji{p_{m},q_{i}}$ with
  $r_{m} < 1$ yields $r_{i} \leq q_{i}$, and
  $\ji{r_{i},r_{M}} \leq \ji{p_{i},q_{M}}$ with $0 < r_{M}$ yields
  $p_{i} \leq r_{i}$. Thus, for such an $i$,
  $p_{i} \leq r_{i} \leq q_{i}$.
  
  Let us verify that (1) and (2) imply $\ji{r} \leq \ji{[p,q]}$.
  Consider $(i,j) \in \cd$: if $i < m$ or $M < j$, then
  $\ji{r_{i},r_{j}} = \bot$, so
  $\ji{r_{i},r_{j}} \leq \ji{p_{i},q_{j}}$ trivially holds; otherwise
  $m \leq i < j \leq M$, and conditions (1) and (2) imply that
  $p_{i} \leq r_{i}$ and $r_{j} \leq q_{j}$.
\end{proof}

As a particular instance of the previous Lemma (i.e. when $p = q$ in
the statement of the Lemma) we deduce the following statement:
\begin{proposition}
  \label{prop:eqleqep}
  Let $r,p\in \I^{d}$ be such that $\bot < \ji{r}$. Then
  $\ji{r} \leq \ji{p}$ if, and only if,
  \begin{enumerate}
  \item $p_{\mji{r}} \leq
    r_{\mji{r}}$, $r_{\Mji{r}} \leq p_{\Mji{r}}$,
  \item  $r_{i} = p_{i}$ for all
    $i \in \setof{d}$ with $\mji{r} < i < \Mji{r}$.
  \end{enumerate}
\end{proposition}
Notice that the relation $\bot < \ji{r} \leq \ji{p}$ also implies that
\begin{align*}
  \mji{p} & \leq \mji{r} < \Mji{r} \leq \Mji{p}\,.
\end{align*}
Suppose for example that $\mji{r} <\mji{p}$, so $p$ has in position
$m := \mji{r}$ a $1$. This implies that $\ji{p_m,p_{j}} = \bot$ for
each $j > m$.  Therefore, $\ji{r} \leq \ji{p}$ implies that
$\ji{r_m,r_{j}} = \bot$ for each $j > m$. Since by definition
$r_{m} < 1$, the relations $\ji{r_m,r_{j}} = \bot$ imply that
$r_{j} = 0$ for each $r_{j} > m$. Yet this means that
$\dimji(r) \leq 0$, so $\ji{r} = \bot$, against the assumption.

\begin{proposition}
  \label{prop:ji2}
  If $\alpha \in \LI[d]$ is join-irreducible, then $\alpha = \ji{p}$
  for some $p \in \I^{d}$.
\end{proposition}
 \begin{proof}
   We claim first that there exists $p \in \I^{d}$ such that $\alpha \leq \ji{p}$.
  To prove the claim, we define an infinite sequence of intervals
  $I_{n} : = [p^{n},q^{n}]$, $n \geq 0$, with the following
  properties:
  \begin{enumerate}
  \item $I_{n +1} \subseteq I_{n}$, for each  $n \geq 0$,
  \item  $q^{n}_{i} - p^{n}_{i} =
    \frac{1}{2^{n}}$, for each $i \in \setof{d}$,
  \item $\alpha = \bigvee \set{\ji{r} \mid \ji{r} \leq \alpha,\,r \in I_{n}}$.
  \end{enumerate}
  Notice that the last condition implies that
  $\alpha \leq \bigvee I_{n}$.

  \smallskip
  
  We let $I_{0} :=\I$, so, for example, (3) holds by
  Corollary~\ref{cor:joinofeps}.

  Given $I_{n}$, we define $I_{n+ 1}$ as follows. For
  each $i \in \setof{d}$, let $I_{i,0} :=
  [p^{n}_{i},\frac{q^{n}_{i}+p^{n}_{i}}{2}]$ and $I_{i,1} :=
  [\frac{q^{n}_{i}+p^{n}_{i}}{2},q^{n}_{i}]$;  given a function $f :
  \setof{d} \rto \set{0,1}$, we let
  \begin{align*}
    I_{f} & := I_{1,f(1)}\times \ldots \times I_{d,f(d)} \,.
  \end{align*}
  Since
  \begin{align*}
    I_{n} & = \bigcup_{f : \setof{d} \rto \set{0,1} } I_{f} 
    \intertext{then} \alpha & = \bigvee \set{\ji{r} \mid \ji{r} \leq
      \alpha, r \in I_{n} } = \bigvee_{f : \setof{d} \rto \set{l,r} }
    \bigvee \set{\ji{r} \mid \ji{r} \leq \alpha, r \in I_{f}}\,,
  \end{align*}
  so, since $\alpha$ is \jirr, there exists $f$ such that
  \begin{align*}
    \alpha & = \bigvee \set{\ji{r} \mid
      \ji{r} \leq \alpha, r \in I_{f}}\,,
  \end{align*}
  We let then $I_{n+1} := I_{f}$.

  Let $\beta_{n} = \bigvee_{n} I_{n}$ and let $p^{\omega}$ be the
  unique element of $\bigcap_{n \geq 0} I_{n}$.  Observe that, since
  the sequences $\set{p^{n}_{i}}_{n\geq 0}$, are increasing while the
  sequences $\set{q^{n}_{i}}_{n\geq 0}$, are decreasing,
  $p^{\omega}_{i} = \bigvee_{n \geq 0} p^{\omega}_{i} = \bigwedge_{n
    \geq 0} q^{n}_{i}$.
  We verify next that
  $\bigwedge_{n \geq 0} \beta_{n} = \ji{p^{\omega}}$.
  Let $r \in \I^{d}$ be such that $r \leq \beta_{n}$ for each
  $n \geq 0$, and put $m := \mji{r}$ and $M := \Mji{r}$.  Then, for
  each $n \geq 0$, $p^{n}_{m} \leq r_{m}$,
  $r_{i} \in [p^{n}_{i},q^{n}_{i}]$, $r_{M} \leq q^{n}_{M}$.  It
  follows that $p^{\omega}_{m} \leq r_{m}$, $r_{i} = p^{\omega}_{i}$ for $i$ such that
  $m < i < M$, that is $\ji{r} \leq \ji{p^{\omega}}$.
  Since $\alpha \leq \beta_{n}$ for each $n\geq 0$ and
  $\alpha = \bigvee \set{\ji{r} \mid \ji{r} \leq \alpha}$, then
  $\alpha \leq \ji{p^{\omega}}$.  This proves our claim.

  \medskip
  Observe now that
  \begin{align*}
    \alpha & = \bigvee_{1 \leq m < M \leq d} J(\alpha,m,M)
    \intertext{where}
    J(\alpha,m,M)
    & = \set{\ji{r} \mid \bot < \ji{r}, \,\mji{r} = m, \,\Mji{r} = M}\,,
  \end{align*}
  so, since $\alpha$ is \jirr, then for some $m,M \in \set{d}$ with $m
  < M$,
  \begin{align*}
    \alpha & = \bigvee J(\alpha,m,M)\,.
  \end{align*}
  
  Observe now that if $r \in \I^{d}$ is such that $\bot < \ji{r}$ and
  $\ji{r} \in J(\alpha,m,M)$, then $\ji{r} \leq \alpha \leq \ji{p}$,
  whence by Lemma~\ref{prop:eqleqep} $r$ is of the form
  \begin{align*}
    r & = (1,\ldots ,1,r_{m},p_{m+1},\ldots ,p_{M-1},r_{M},0,\ldots,0)\,. 
  \end{align*}
  The join $\alpha = \bigvee J(\alpha,m,M)$ is then $\ji{q}$ with
  \begin{align*}
    q & = (1,\ldots ,1,\bigwedge_{e_{r} \in J(\alpha,m,M)}
    r_{m},p_{m+1},\ldots ,p_{M-1},\bigvee_{\ji{r} \in J(\alpha,m,M)}
    r_{M},0,\ldots,0)\,.
    \tag*{\qedhere}
  \end{align*}
\end{proof}

\medskip

\subsection{Lack of compact elements}
\label{subsec:lackspatials}
Let $L$ be a complete lattice. An element $j$ of $L$ is \emph{\cjirr}
if, for any $X \subseteq L$, $j = \bigvee X$ implies $j \in X$; it is
\emph{\cjp} if, for any $X \subseteq L$, $j \leq \bigvee X$ implies
$j \in x$, for some $x \in X$. Every \cjp element is also \cjirr. If
$L$ is a frame, that is, if
$x \land \bigvee Y = \bigvee_{y \in Y} x \land y$ for each $x \in L$
and $Y \subseteq L$, then the converse holds as well.

A family $\cF \subseteq L$ is
\emph{directed} if every finite (possibly empty) subset of $\cF$ has
an upper bound in $\cF$.  An element $c \in L$ is \emph{compact} if,
for every directed family $\cF \subseteq L$, $c \leq \bigvee \cF$
implies $c \leq f$ for some $f \in \cF$.

Let us remark that there are no \cjp (equivalently, \cjirr) elements
in $\QjI$. Indeed, for every prime \osf $\ji{x,y}$, we can write
$\ji{x,y} = \bigvee_{\ell \in L} \ji{x_{\ell},y_{\ell}}$ where the set
$\set{\ji{x_{\ell},y_{\ell}} \mid \ell \in L}$ is a chain and
$\ji{x_{\ell},y_{\ell}} < \ji{x,y}$, for each $\ell \in L$.
Similarly, there are no compact elements in $\LjI$. Indeed, if $f$ is
compact, then Proposition~\ref{prop:generationDim2} implies that $f$
is a finite join of \jirr elements below it, say
$f = \bigvee_{i=1,\ldots ,n}\ji{x_{i},y_{i}}$. We can assume that
$\set{\ji{x_{i},y_{i}} \mid i=1,\ldots ,n}$ is an antichain.  Now, if
$\set{ \ji{x_{1,\ell},y_{1,\ell}} \mid \ell \in L}$ is a chain
approximating strictly from below $\ji{x_{1},y_{1}}$, then
$f = \bigvee_{\ell \in L} \ji{x_{1,\ell}y_{1,\ell}} \vee
\bigvee_{i=2,\ldots ,n}\ji{x_{i},y_{i}}$, so
$f = \ji{x_{1,\ell}y_{1,\ell}} \vee \bigvee_{i=2,\ldots
  ,n}\ji{x_{i},y_{i}}$ for some $\ell \in L$. It follows that
$\ji{x_{1},y_{1}} \leq f = \ji{x_{1,\ell}y_{1,\ell}} \vee
\bigvee_{i=2,\ldots ,n}\ji{x_{i},y_{i}}$, so either
$\ji{x_{1},y_{1}} \leq \ji{x_{1,\ell}y_{1,\ell}}$, or
$\ji{x_{1},y_{1}} \leq \ji{x_{i},y_{i}}$ for some $i = 2,\ldots
,n$. In all the cases we obtain a contradiction.

For a similar reason, the lattices $\LId$ have no \cjirr
elements. Indeed, given $p \in \Id$ such that $ \bot < \ji{p}$, it is
easy to construct (using Proposition~\ref{prop:eqleqep}) a chain of
\jirr elements strictly below $\ji{p}$ whose join is $\ji{p}$.

In the rest of this section we argue that the lattices $\LId$ do not
have any compact element.

\begin{lemma}
  \label{lemma:algebraic}
  Let $\cF \subseteq \QjI^{\cd}$ be a directed family of closed
  elements. Then then join $\bigvee_{\QjI^{\cd}} \cF \in \QjI^{\cd}$ is closed.
\end{lemma}
\begin{proof}
  A straightforward verification:
  \begin{align*}
    (\bigvee_{f \in \cF} f)_{j,k} \circ (\bigvee_{f \in \cF} f)_{i,j}
    & = \bigvee_{f,g \in \cF} g_{j,k} \circ f_{i,j}
    \\
    & \leq \bigvee_{h \in \cF} h_{j,k} \circ h_{i,j}\,,
    \tag*{using the fact that $\cF$ is directed,} \\
    & \leq \bigvee_{h \in \cF} h_{i,k} = (\,{\bigvee\textstyle{{\!\!}_{\QjI^{\cd}}}} \cF\,)_{i,k}\,. 
    \tag*{\qedhere}
  \end{align*}
\end{proof}
\begin{proposition}
  The lattice $\LId$ has no compact elements.
\end{proposition}
\begin{proof}
  Suppose $c \in \LId$ is a compact element.
  Recall from  Lemma~\ref{lemma:approx} and
  Corollary~\ref{cor:joinofeps} that we can write
  \begin{align*}
    c & = \bigvee_{x \in \I} \bigvee_{1 \leq m < M \leq d}
    \ji{p(c,x,m,M)}\,,
  \end{align*}
  where $p:=p(c,x,m,M)$ is such that $m = \mji{p}$, $M =
  \Mji{p}$, $p_{m} = x$ and $p_{i} = \Meetof{c_{m,i}}(x)$, for $i =
  m+1,\ldots ,M$.  Since $c$ is compact, there exists a finite set $P
  \subseteq \set{p(c,x,m,M) \mid x \in \I, 1 \leq m < M \leq
    d}$ such that $c = \bigvee \set{\ji{p} \mid p\in
    P}$ and we can suppose that $P$ is an antichain.  Let $p^{1} \in{
    P} $ be such that $m := \mji{p^{1}}$ is minimal in $\set{\mji{p}
    \mid p \in P}$ and such that $p^{1}_{m}$ is minimal in $\set{p_{m}
    \mid p \in P, \mji{p} = m}$.  Let therefore $\set{p_{1},\ldots
    ,p_{n}} := P$.

  \begin{claim}
    For each $x \in \I$ such that $p^{1}_{m} < x \leq 1$, define
    $p^{1}_{x}$ by $p^{1}_{x,m} := x$ and $p^{1}_{x,i} := p^{1}_{i}$
    for $i \neq m$. Then
    $\ji{p^{1}_{x}} \vee \bigvee_{i=2,\ldots ,n} \ji{p^{i}} <
    \bigvee_{i=1,\ldots ,n} \ji{p^{i}}$.
  \end{claim}
  To ease reading of the proof of the claim, let $q_{1} := p^{1}_{x}$
  and let also $q^{i} := p^{i}$, for $i = 2,\ldots ,n$; notice that
  $q^{1} < p^{1}$.  
  Suppose
  $\bigvee_{i=1,\ldots ,n} \ji{q^{i}} < \bigvee_{i=1,\ldots ,n}
  \ji{p^{i}}$ does not hold. Then
  $\ji{p^{1}} \leq c = \bigvee_{i=1,\ldots ,n} \ji{q^{i}}$; by the
  formula for the join in equation \eqref{eq:defJoinAndMeet} and by
  Lemma~\ref{lemma:ClosureInterior}, there exists a subdivision
  $m = \ell_{0} < \ell_{1} < \ldots < \ell_{k} = M$ of the interval
  $[m,M]$ such that
  \begin{align*}
    \ji{p^{1}_{m},p^{1}_{M}}& \leq (\bigvee_{i=1,\ldots, n}
    \ji{q^{i}_{\ell_{k-1}},q^{i}_{\ell_{k}}}) \circ \ldots \circ
      (\bigvee_{i=1,\ldots, n}
      \ji{q^{i}_{\ell_{0}},q^{i}_{\ell_{1}}})\,.
    \end{align*}
    Considering that composition distributes over joins and that
    $\ji{p_{m},p_{M}}$ is \jirr in $\QjI$, for each
    $u \in \set{0,\ldots ,k-1}$, there exists
    $i_{u} \in \set{1,\ldots ,n}$ such that
    \begin{align*}
      \ji{p^{1}_{m},p^{1}_{M}}& \leq
      \ji{q^{i_{k-1}}_{\ell_{k-1}},q^{i_{k-1}}_{\ell_{k}}} \circ
      \ldots \circ \ji{q^{i_{0}}_{\ell_{0}},q^{i_{0}}_{\ell_{1}}}\,.
    \end{align*}
    By Lemma~\ref{lemma:compJI} and since
    $\bot \neq \ji{p_{m},p_{M}}$, the expression above on the right equals
    to
    $\ji{q^{i_{0}}_{\ell_{0}},q^{i_{k-1}}_{\ell_{k}}} =
    \ji{q^{i_{0}}_{m},q^{i_{k-1}}_{M}}$, so
    \begin{align*}
      \ji{p^{1}_{m},p^{1}_{M}}& \leq  \ji{q^{i_{0}}_{m},q^{i_{k-1}}_{M}}\,,
    \end{align*}
    which, by Corollary~\ref{cor:orderonjp}, amounts to
    $q^{i_{0}}_{m} \leq p^{1}_{m}$ and
    $p^{1}_{M} \leq q^{i_{k-1}}_{M}$. Since
    $p^{1}_{m} < x = q^{1}_{m}$, $i_{0} \neq 1$.  It also implies that
    $\mji{p^{i_{0}}} = \mji{q^{i_{0}}} \leq \mji{p^{1}} = m$ and,
    considering the minimality of $\mji{p^{1}}$,
    $\mji{p^{i_{0}}} = \mji{p^{1}} = m$. Since $p^{1}_{m}$ is minimal
    among elements of the form $p^{i}_{m}$, $i = 1,\ldots ,n$, we also
    infer that $p^{i_{0}}_{m} = p^{1}_{m}$.  Yet, this implies that,
    for $i = m+1,\ldots , \min(\Mji{p^{1}},\Mji{p^{i_{0}}})$,
    \begin{align*}
      p^{1}_{i} &  = \MeetOf{c_{m,i}}(p^{1}_{m}) = \MeetOf{c_{m,i}}(p^{1}_{m}) =
      p^{i_{0}}_{i}\,.
    \end{align*}
    Using Proposition~\ref{prop:eqleqep}, it immediately follows that
    $p^{1}$ and $p^{i_{0}}$ are comparable, contradicting the
    assumption that $P$ is an antichain. This ends the proof of the
    Claim.

    \medskip
    Clearly, the following relations holds in $\LId$:
    \begin{align*}
      \bigvee_{1 \geq x > p^{1}_{m}} (\,\ji{p^{1}_{x}} \vee
      \bigvee_{i=2,\ldots ,n} \ji{p^{i}}\,) & \leq c \,.
    \end{align*}
    Let us argue that also the converse inclusion holds.  Within
    $\QjI^{\cd}$, the following relation holds:
    \begin{align*}
      \bigvee_{i = 1,\ldots ,n} \ji{p^{i}} & \leq \bigvee_{1 \geq x >
        p^{1}_{m}} (\,\ji{p^{1}_{x}} \vee \bigvee_{i=2,\ldots ,n} \ji{p^{i}}\,)\,.
    \end{align*}
    Taking the closure, we have
    \begin{align*}
      \closure{\bigvee_{i = 1,\ldots ,n} \ji{p^{i}}} \ & \leq \closure{\bigvee_{1 \geq x >
        p^{1}_{m}} (\,\ji{p^{1}_{x}} \vee \bigvee_{i=2,\ldots ,n}
      \ji{p^{i}}\,)}
    = \bigvee_{1 \geq x >
        p^{1}_{m}} \closure{(\,\ji{p^{1}_{x}} \vee \bigvee_{i=2,\ldots ,n}
      \ji{p^{i}}\,)}
    \,.
  \end{align*}
  where in the last equality we have used the fact that
  $\set{\closure{\ji{p^{1}_{x}} \vee \bigvee_{i=2,\ldots ,n}
      \ji{p^{i}}} \mid 1 \geq x > p_{m}}$ is a directed set and
  Lemma~\ref{lemma:algebraic}.

  Thus $c$ is, within $\LId$, an infinite join of a chain of elements
  that are strictly below it. This contradicts $c$ being compact.
\end{proof}

\section{Embeddings from multinomial lattices}
\label{sec:embeddings}

The goal of this section is study how multinomial lattices
\cite{BB,ORDER-24-3} embed into the lattices $\LId$. We proceed by
arguing that these lattices are of the form $\Ld{Q}$ for some \msaq
$Q$ such that $Q$ embeds, as an \lbs, into $\Qj(\I)$. By functoriality
(Proposition~\ref{prop:functor}), it follows that $\Ld{Q}$ embeds into
$\LId = \Ld{\Qj(\I)}$.

\medskip

In the following let $I_{0},I_{1}$ be two complete chains and let
$\iota : I_{0} \rto I_{1}$ be a \bc 
embedding (thus we ask that $\iota$ preserves all joins and all meets;
in particular, $\iota$ preserves the bounds of the chains). Then
$\iota$ has both a \LADJ $\ell : I_{1} \rto I_{0}$ and a \RADJ
$\rho : I_{1} \rto I_{0}$. It useful, e.g. when $I_{0}$ is finite, to
think of $\ell$ as the ceiling function and of $\rho$ as the floor
function.
\begin{lemma}
  \label{lemma:iotalellrho}
  The following statements hold:
  \begin{itemize}
  \item $\ell \circ \iota = \rho \circ \iota = id_{I_{0}}$,
  \item $\rho \leq \ell$,
  \item if $y \in I_{1}$ is such that $\rho(y) = \ell(y)$, then
    $y = \iota(x)$, with $x = \ell(y) \in I_{0}$.
  \end{itemize}
\end{lemma}
\begin{proof}
  By standard laws of adjunctions, $\iota(x)= \iota(\ell(\iota(x))$,
  for each $x \in I_{0}$. Since $\iota$ is an embedding, we deduce
  $x = \ell(\iota(x))$. The equality $x = \rho(\iota(x))$ is proved
  similarly.
  
  Let now $y \in I_{1}$ and suppose that $\ell(y) \leq \rho(y)$, then
  we have $y \leq \iota(\ell(y))$ as unit of the adjunction,
  $\iota(\ell(y)) \leq \iota(\rho(y))$ and $\iota(\rho(y)) \leq y$ as
  counit of the adjunction. Therefore
  $y = \iota(\ell(y)) = \iota(\rho(y))$ and $\ell(y) = \rho(x)$, since
  $\iota$ is an embedding.

  From this it follows that, for $y \in I_{1}$, then either
  $\rho(y) = \ell(y)$, in which case $y = \iota(\ell(y))$, or
  $\rho(y) \neq \ell(y)$, in which case we cannot have
  $\ell(y) \leq \rho(y)$, so $\rho(y) < \ell(y)$.
\end{proof}
\begin{lemma}
  If
  $\ell(y) < x$, then $y < \iota(x)$.
\end{lemma}
\begin{proof}
  Assume $\ell(y) < x$. From $\ell(y) \leq x$ we deduce
  $y \leq \iota(x)$. If the latter inclusion is not strict, then
  $\iota(x) \leq y$ and $x \leq \rho(y)$, so
  $\ell(y) < x \leq \rho(y)$ yields the relation $\ell(y) < \rho(y)$,
  which contradicts $\rho \leq \ell$ established in
  Lemma~\ref{lemma:iotalellrho}.
\end{proof}

For each monotone $f : I_{0} \rto I_{0}$, define
$\Ri(f) : I_{1} \rto I_{1}$ by the formula
\begin{align*}
  \Ri(f) & := \iota \circ f \circ \ell \,. 
\end{align*}
Figure~\ref{fig:KanExt} gives some hints on the geometric meaning of
the correspondence $f \mapsto \Ri(f)$.  In the figure we have
$I_{0} = \I_{4} = \set{0,1,2,3,4,5}$, $I_{1} = \I$, and $f(0) = 0$,
$f(1) = f(2) = 2$, $f(3) = f(4) = 3$. In some sense, this
correspondence the responsible for representing \jc functions from
some $\I_{n}$ to itself as discrete paths in the plane.  In the
figure, the graph of the function $\Ri(f)$ (in blue) is completed with
the vertical intervals (in red), so to yield the path $C_{f}$,
similarly to what we have done in Figure~\ref{fig:Cf}. From the figure
it should also be clarified the recipe
$\Ri(f)(x) = (\iota \circ f \circ \ell)(x)$: give to $x$ the same
value of its ceiling $\ell(x)$ and then inject back this value back
into $\I$ using $\iota$.
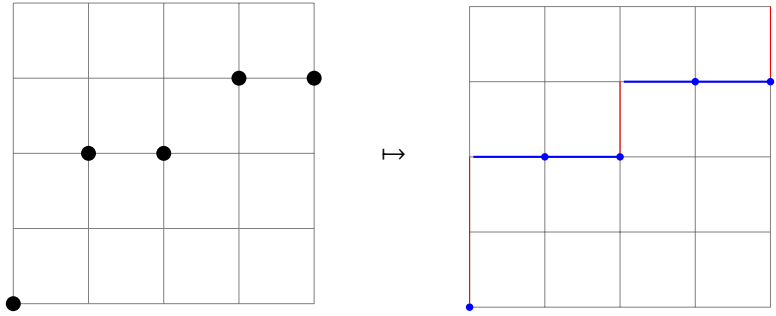
\begin{figure}
  \centering
  \begin{tikzpicture}
    \draw[step=1cm,gray,very thin] (0,0) grid (4,4);
    \bpoint{0,0};
    \bpoint{1,2};
    \bpoint{2,2};
    \bpoint{3,3};
    \bpoint{4,3};
  \end{tikzpicture}
  \raisebox{2cm}{$\qquad\mapsto\qquad$}
  \begin{tikzpicture}
    \draw[step=1cm,gray,very thin] (0,0) grid (4,4);
    \lpoint{0,0};
    \lpoint{1,2};
    \lpoint{2,2};
    \lpoint{3,3};
    \lpoint{4,3};
    \draw [thin,red] (0,0.05) -- (0,2) ; 
    \draw [thick,blue] (0.05,2) -- (2,2) ;  
    \draw [thin,red] (2,2.05) -- (2,3) ;
    \draw [thick,blue] (2.05,3) -- (4,3) ;  
    \draw [thin,red] (4,3.05) -- (4,4) ;
  \end{tikzpicture}
  \caption{The correspondence sending $f$ to $\Ri(f)$}
  \label{fig:KanExt}
\end{figure}
\begin{lemma}
  \label{lemma:Kanextension}
  For each monotone $h : I_{1} \rto I_{1}$,
  $h \circ i \leq \iota \circ f$ if and only if $h \leq \Ri(f)$.  That
  is, $\Ri(f)$ is the right Kan extension of
  $\iota \circ f : I_{0} \rto I_{1}$ along $\iota : I_{0} \rto I_{1}$.
\end{lemma}
\begin{proof}
  Indeed, observe that
  $\Ri(f) \circ \iota = \iota \circ f \circ \ell \circ \iota = \iota
  \circ f$. 
  Next, if $h \circ \iota \leq \iota \circ f$, then
  $h \leq h \circ \iota \circ \ell \leq \iota \circ f \circ \ell = \Ri(f)$.
\end{proof}

\begin{proposition}
  $\Ri$ is injective and restricts to a map from $\Lj(I_{0})$ to
  $\Lj(I_{1})$.
\end{proposition}
\begin{proof}
  $\Ri$ is injective since $\iota$ is monic and $\ell$ is epic. For
  the second statement, notice that if $f \in \Lj(I_{0})$, then
  $\Ri(f) \in \Lj(I_{1})$, since $\Ri(f)$ is the composition of three
  \jc maps.
\end{proof}
We shall observe next that $\Ri$ preserves part of the structure of
$\Lj(I_{0})$, $\otimes,\oppfun,\oplus$, as well as finite meets and
infinite joins. On the other hand, it is easily seen that units are
only semi-preserved.

\begin{proposition}
  For each $f,g \in \Qj(I_{0})$,
  the following relation holds
  \label{prop:tensorpreserved}
  \begin{align*}
    \Ri(f \otimes g) & = \Ri(f) \otimes \Ri(g)\,, 
    \qquad \Opp{\Ri(f)}  = \Ri(\Opp{f})\,,
    \intertext{and, consequently,}
    \Ri(f \oplus g) & = \Ri(f) \oplus \Ri(g)\,.
  \end{align*}
\end{proposition}
\begin{proof}
  For the first relation we compute as follows:
  \begin{align*}
    \Ri(f) \otimes \Ri(g) & = \Ri(g) \circ \Ri(f) =
    \iota \circ g \circ \ell
    \circ \iota \circ f \circ \ell = \iota \circ g \circ f \circ \ell \\
     &= \Ri(g \circ f) =
    \Ri(f \otimes g) \,.  
  \end{align*}
  For the second relation, we first establish that  $\Opp{\Ri(f)} \leq
  \Ri(\Opp{f})$.
  In view of Lemma~\ref{lemma:Kanextension},
  it is enough to prove
  $\Opp{\Ri(f)} \circ \iota \leq \iota \circ \opp{f}$. This is
  accomplished as follows:
  \begin{align*}
    \Opp{\Ri(f)}(\iota(x)) & = \bigvee \set{y \in I_{1}\mid (\iota
      \circ f)(\ell(y))
      < \iota(x)}\,,
    \tag*{by equation~\eqref{eq:opp},}\\
    & = \bigvee \set{y \in I_{1}\mid f(\ell(y)) < x}\,, \tag*{since
      $\iota$ is an embedding,}
    \\
    & \leq \bigvee \set{\iota(x') \mid x'\in I_{0}\,,\;f(x') < x}\,,
    \intertext{
      since if $y \in I_{1}$ is such that $f(\ell(y)) < x$, then, by
      letting $x' := \ell(y)$, $f(x') < x$ and
      $y \leq \iota(\ell(y)) = \iota(x')$,} & = \iota (\bigvee \set{x'
      \in I_{0}\mid f(x') < x}) = (\iota \circ \opp{f})(x)\,.
   \end{align*}
   Next we establish that $\Ri(0) \leq 0$.
  Let us recall that, for each $y \in I_{1}$,
  \begin{align*}
    \Ri(0)(y) & = \bigvee_{x < \ell(y)} \iota(x)\,,
    \quad
    0(y)  = \bigvee_{z < y} z\,.
  \end{align*}
  Therefore, to prove $\Ri(0) \leq 0$, it is enough to argue that
  $x < \ell(y)$ implies $\iota(x) < y$.  Now, if $x < \ell(y)$, then
  $\ell(y) \not\leq x$, so $y \not\leq \iota(x)$, that is,
  $\iota(x) < y$.

  We can now argue that $\Ri(\opp{f}) \leq \opp{\Ri(f)}$.
  This relation is equivalent to
  $\Ri(\opp{f})\otimes \Ri(f) \leq 0$ which can be derived as follows:
  \begin{align*}
    \Ri(\opp{f})\otimes \Ri(f)
    & = \Ri(\opp{f} \otimes f) 
    \leq \Ri(0) \leq 0\,.
  \end{align*}
  Therefore $\opp{\Ri(f)} = \Ri(\opp{f})$.
  For the last statement, recall that
  $f \oplus g = \opp{(\opp{g} \otimes \opp{f})}$, so preservation of
  $\oplus$ follows from preservation of $\otimes$ and $\oppfun$.
\end{proof}
\begin{proposition}
  We have
  \begin{align*}
    \Ri(\bigvee_{i \in I} f_{i})
    & = \bigvee_{i \in I} \Ri(f_{i})\,,
    &
    \Ri(\bigwedge_{i = 1,..,n}f_{i}) & = \bigwedge_{i = 1,..,n}\Ri(f_{i})\,.
  \end{align*}
\end{proposition}
\begin{proof}
  \begin{align*}
    \Ri(\bigvee_{i \in I} f_{i})(x)
    & = \iota( (\bigvee_{i \in I} f_{i})(\ell(x)))
    = \iota( \bigvee_{i \in I} (f_{i}(\ell(x))) \,)
    = \bigvee_{i \in I}\iota(f_{i}(\ell(x))) =
    (\bigvee_{i \in I} \Ri(f_{i}))(x)\,.
  \end{align*}
  In a similar way, considering that finite meets in $\Lj(I)$ are
  computed pointwise, we have
  \begin{align*}
    \Ri(\bigwedge_{i \in I} f_{i})(x)
    & = \iota( (\bigwedge_{i \in I} f_{i})(\ell(x)))
    = \iota( \bigwedge_{i \in I} (f_{i}(\ell(x))) \,)
    = \bigwedge_{i \in I}\iota(f_{i}(\ell(x))) =
    (\bigwedge_{i \in I} \Ri(f_{i}))(x)\,.
    \tag*{\qedhere}
  \end{align*}
\end{proof}
We can state now our main result.
\begin{theorem}
  \label{thm:main:sec:embeddings}
  For each pair of perfect chains $I_{0},I_{1}$ and each \bc embedding
  $\iota : I_{0} \rto I_{1}$, the map
  $\Ri : \Qj(I_{0}) \rto \Qj(I_{1})$ is an \lbs embedding.
  Together with $\Ri[(\_)]$, $\Qj{(\_)}$ is a functor from the category of 
  perfect chains and \bc embeddings to the category of  \lbs{s}.
\end{theorem}
\begin{proof}
  The first statement of the Theorem just summarizes the observations
  made up to now.
  The expression $\Ri$ is functorial in $\iota$, since if
  $\iota = \iota_{2} \circ \iota_{1}$, then
  $\ladj{\iota} = \ladj{(\iota_{1})} \circ
  \ladj{(\iota_{2})}$. Therefore
  \begin{align*}
    \Ri[\iota_{2} \circ \iota_{1}](f)
    & = \iota_{2} \circ \iota_{1} \circ f \circ \ladj{(\iota_{1})}
    \circ \ladj{(\iota_{2})} = \iota_{2} \circ \Ri[\iota_{1}](f) \circ
      \ladj{(\iota_{2})}
        = \Ri[\iota_{2}](\Ri[\iota_{1}](f))\,.
      \end{align*}
      In a similar way, $\Ri[id_{I_{0}}] = id_{\Lj(I_{0})}$.
\end{proof}

\begin{definition}
  For each $n \geq 1$ and each $x \in \In$, define
  $j_{n}(x) := \frac{x}{n} \in \I$. For each $n, m \geq 1$ and each
  $x \in \In$, let $j_{n,m}(x) := mx \in \In[nm]$.
\end{definition}
Clearly, $j_{n}$ and $j_{n,m}$ are complete embeddings; observe also
that
\begin{align*}
  j_{mn}(j_{n,m}(x)) & = \frac{mx}{nm} = j_{n}(x)\,,
\end{align*}
and that
\begin{fact}
  The diagram $j_{n,m} : \In \rto \Im$ is directed and
  $j_{n} : \In \rto \I$ is a cocone.
\end{fact}

The following statement is a consequence of functoriality of the
constructions $\Ri[(\_)]$ and $\Ld{\_}$, see
Proposition~\ref{prop:functor} and
Theorem~\ref{thm:main:sec:embeddings}.
\begin{proposition}
  The diagram $\Ri[j_{n,m}] : \Lj(\In) \rto \Lj(\Im), m \geq n \geq 1$
  is directed and $\Ri[j_{n}] : \Lj(\In) \rto \Lj(\I)$ is a cocone.
  For each $d \geq 2$, there is a directed diagram in the category of
  lattices
  $\Ld{\Ri[j_{n,m}]} : \Ld{\Lj(\In)} \rto \Ld{\Lj(\Im)}, m \geq n \geq
  1$ is directed and
  $\Ld{\Ri[j_{n}]} : \Ld{\Lj(\In)} \rto \Ld{\Lj(\I)}$ is a cocone.
\end{proposition}

\begin{definition}
  \label{def:LRId}
  We let $\LRId$ be the image of all the mappings
  $\Ld{\Ri[j_{n}]} : \Ld{\Lj(\In)} \rto \Ld{\Lj(\I)}$. 
\end{definition}
By general facts, $\LRId$ yields an explicit representation of the
colimit of the directed diagram
$\Ld{\Ri[j_{n}]} : \Ld{\Lj(\In)} \rto \Ld{\Lj(\I)}$; in particular it
is a sublattice of $\LId$.  Observe that $f \in \LRId$ if and only if
$f$ is clopen and, for each $(i,j) \in \cd$, $f_{i,j}$ is a finite
join of rational \osf{s}.

\section{Generation from rational one step functions}
\label{sec:generation}
As a first application of the characterization of \jirr elements and
of their order, we show that if $d \geq 3$ $\LId$ is not the
Dedekind-MacNeille completion of $\LRId$, see
definition~\ref{def:LRId}.  This is the sublattice if $\LId$ of those
$f \in \LId$ such that each $f_{i,j}$ is a finite join of rational
one-step functions.  This contrasts with the case where $d = 2$, when
$\LId = \QjI$, cf. Remark~\ref{remark:QjIcompletion}.

\begin{theorem}
  For $d \geq 3$, the lattice $\LI[d]$ is not (isomorphic to) the
  Dedekind-MacNeille completion of $\LRId$.
\end{theorem}
\begin{proof}
  We need to find an element of $\LId$ which is not an infinite join
  of elements of $\LRId$.
  For example, let $d = 3$ and choose $p \in \I^{3}$ such that
  $p_{1} < 1$, $p_{2}$ is irrational, and $0 < p_{3}$ (so
  $\mji{p} = 1$ and $\Mji{p} = 3$). If $\ji{p}$ can be written as an
  infinite join of elements from $\LRId$, then it can also be written
  as an infinite join of join-irreducible elements from $\LRId$ below
  it, and these are of the form $\ji{r}$ with $r \in (\IQ)^{3}$.  We
  can therefore write
  \begin{align*}
    \ji{p} & = \bigvee \set{\ji{r} \in \LRId[3] \mid \bot <\ji{r} \leq \ji{p} }
    =
    \bigvee_{(i,j) \in \couples{3}}
    \bigvee \JR(\ji{p},i,j) \,,
    \intertext{where}
    \JR(\ji{p},i,j) & := \set{\ji{r} \mid r \in (\IQ)^{3},\,\bot <\ji{r} \leq
      \ji{p}, \,\mji{r} = i,\, \Mji{r} = j }\,.
  \end{align*}
  Since $\ji{p}$ is \jirr, then we have
  $\ji{p} = \bigvee \JR(\ji{p},i,j)$ for some $(i,j) \in \couples{3}$.
  If $(i,j) = (1,2)$, then we deduce that $p_{3} = 0$, and if
  $(i,j) = (2,3)$, then we deduce that $p_{1} = 1$; these are
  contradictions.  Therefore we have $(i,j) = (1,3)$.  Yet, by
  Proposition~\ref{prop:eqleqep}, $\JR(\ji{p},1,3) = \emptyset$, since
  if $\bot < \ji{r} \leq \ji{p}$, then $r_{2} = p_{2}$ is
  irrational. We deduce therefore $\ji{p} = \bot$, a contradiction.
\end{proof}

To understand how the lattice $\LId$ is generated from $\LRId$, we
need to study its \mirr elements.
For $x,y \in \I$, we define $\mi{x,y} \in \LjI$ as follows:
\begin{align*}
  \mi{x,y}(t) & =
  \begin{cases}
    0 \,, & t = 0\,, \\
    y & 0 < t \leq x\,,\\
    1 & x < t \leq 1\,.
  \end{cases}
\end{align*}
Observe that $\mi{x,y} = \Opp{\ji{y,x}}$ and, therefore, \mirr
elements of $\LjI$ are, by duality, exactly those of the form
$\mi{x,y}$ for $x,y \in \I$ such that $0 < x$ and $y < 1$. Notice also
that
\begin{align}
  \label{eq:mirrfromjirr}
  \mi{x,y} & =  \ji{0,y} \vee \ji{x,1} \,.
\end{align}
Let now $d \geq 3$; for each $p \in \I^{d}$, let in the following
\begin{align*}
  \mi{p} & := \Fam{\mi{p_{i},p_{j}} \mid (i,j) \in \cd}\,,
\end{align*}
as well as
\begin{align*}
  \mmi_{p} &:= \min \set{i \in [d] \mid 0 < p_{i} }\,, 
  &\Mmi_{p} & := \max \set{j \in [d] \mid p_{j} < 1}\,, 
\end{align*}
where, by convention, $\min \emptyset = d+1$ and
$\max \emptyset = 0$. For $p \in \I^{d}$, let
\begin{align*}
  \dimmi(p) := \Mmi_{p} - \mmi_{p}\,.
\end{align*}
Notice that, since we assume $d \geq 1$, we cannot have $\Mmi_{p} = 0$
and $\mmi_{p} = d +1$, so $\dim(\mi{p}) \in \set{-d,\ldots, d}$.

\begin{proposition}
  The \mirr elements of $\L(\I^{d})$ are exactly the elements of the
  form $\mi{p}$ for some $p \in \I^{d}$ such that $\dimmi(p) > 0$.
\end{proposition}
\begin{proof}
  It is enough to verify that these elements of $\LId$ correspond,
  under the duality, to \jirr elements.  Indeed we have
  \begin{align*}
    \Opp{\mi{(p_{1},\ldots ,p_{d})}} & =
    \Fam{\Opp{\mi{p_{\sigma(j)},p_{\sigma(i)}}} \mid (i,j) \in \cd } \\
    & =
    \Fam{\ji{p_{\sigma(i)},p_{\sigma(j)}}} = e_{p_{d},\ldots
      ,p_{1}}\,.
  \end{align*}
  Moreover, writing $\sigma(p)$ for $(p_{d},\ldots ,p_{1})$, we have
  $\dimmi(p) = \dimji(\sigma(p))$. The statement of the proposition
  follows now by the previous characterization of \jirr elements of
  $\LId$, see Propositions~\ref{prop:jipji} and~\ref{prop:ji2}.
\end{proof}

We find next an analogous of equation~\eqref{eq:mirrfromjirr} for
higher dimensions. Such an analogous will allow us to argue that every
$f \in \LId$ is a meet of joins (and, dually, a join of meets) of
elements from $\LRId$.
Let $\Mi{x}{y}{i}{j} \in \PrLI$ be the tuple that has $\mi{x,y}$ in
coordinate $(i,j)$ and $\bot$ in the other coordinates. Similarly,
$\JJi{x}{y}{i}{j} \in \PrLI$ denotes the tuple that has $\ji{x,y}$ in
coordinate $(i,j)$ and $\bot$ in the other coordinates.
The
following relations hold  within
$\PrLI$:
\begin{align*}
  \mi{p} & = \bigvee_{(i,j) \in \cd} \Mi{p_{i}}{p_{j}}{i}{j} \,, &
  \Mi{p_{i}}{p_{j}}{i}{j} & = \JJi{0}{p_{j}}{i}{j} \vee
  \JJi{p_{i}}{1}{i}{j} \,,
\end{align*}
and therefore
\begin{align*}
  \mi{p} & = \bigvee_{(i,j) \in \cd} \Mi{p_{i}}{p_{j}}{i}{j} =  \bigvee_{(i,j) \in
    \cd} \JJi{0}{p_{j}}{i}{j}
  \vee
  \bigvee_{(i,j) \in \cd} \JJi{p_{i}}{1}{i}{j}  =  \G{p} \vee \F{p} \,,
  \intertext{where}
  \G{p} & := \bigvee_{(i,j) \in \cd} \JJi{0}{p_{j}}{i}{j}\,
  \quad \tand
  \quad
  \F{p} := \bigvee_{(i,j) \in \cd} \JJi{p_{i}}{1}{i}{j}  .
\end{align*}

\begin{lemma}
  For each $p \in \I^{d}$, both $\G{p}$ and $\F{p}$ belong to $\LId$.
\end{lemma}
\begin{proof}
  We firstly consider $\G{p}$, observing that
  $\G{p} = \Fam{f_{i,j} \mid (i,j) \in \cd}$ with
  $f_{i,j} = e_{0,p_{j}}$.  We argue that $\G{p}$ is clopen relying
  on Remark~\ref{rem:suffCondClopen}.  Let $i,j,k \in \setd$ with
  $i < j < k$.  If $0 < p_{j}$, then
  $f_{j,k} \circ f_{i,j}= \ji{0,p_{k}} \circ \ji{0,p_{j}} =
  \ji{0,p_{k}} = f_{i,k}$, by Lemma~\ref{lemma:compJI}.  If
  $p_{j} = 0$, then
  $\MeetOf{f_{j,k}} \circ \MeetOf{f_{i,j}}= \Ji{0,p_{k}} \circ
  \Ji{0,p_{j}} = \Ji{0,p_{k}} \circ \Ji{0,0} = \Ji{0,p_{k}} =
  \MeetOf{f_{i,k}}$, by Lemma~\ref{lemma:compJIbis}.

  Next, we observe that $\F{p} = \Fam{f_{i,j} \mid (i,j) \in \cd}$
  with $f_{i,j} = e_{p_{i},1}$.  We
  use again
  Remark~\ref{rem:suffCondClopen} to verify that $\F{p}$ is clopen.
  Let $i,j,k \in \setd$ with $i < j < k$; if $p_{j} < 1$, then
  $f_{j,k} \circ f_{i,j} = \ji{p_{j},1} \circ \ji{p_{i},1} =
  \ji{p_{i},1} = f_{i,k}$, by Lemma~\ref{lemma:compJI}; if
  $p_{j} = 1$, then
  $\MeetOf{f_{j,k}}\circ\MeetOf{f_{i,j}} =\Ji{p_{j},1} \circ
  \Ji{p_{i},1} = \Ji{1,1} \circ \Ji{p_{i},1} = \Ji{p_{i},1} =
  \MeetOf{f_{i,k}}$, using Lemma~\ref{lemma:compJIbis}.
\end{proof}
\begin{remark}
  \label{rem:joins}
  Let $L$ be a complete lattice and let $M$ be a subset of $L$ which
  is itself a complete lattice w.r.t. the order inherited from $L$. If
  $Q \subseteq M$, $q \in M$ and the relation $\bigvee Q = q$ holds in
  $L$, then the same relation holds in $M$.
\end{remark}
In view of the remark, we have achieved generalizing
equation~\ref{eq:mirrfromjirr} to higher dimensions:
\begin{corollary}
  The relation 
  \begin{align*}
    \mi{p} & = \G{p} \vee \F{p}  
  \end{align*}
  holds in $\LId$. 
\end{corollary}
For $i \in \setd$, $x \in \I$ and $y \in \set{0,1}$, let us use
$\myvec{i,x,y}$ to denote the point of $\I^{d}$ that has $x$ in
position $i$ and $y$ in all the other coordinates. For an example with
$d =3$, consider $\myvec{2,x,1} = (1,x,1)$; notice that
$\ji{\myvec{2,x,1}} = \langle \ji{1,x},\ji{1,1},\ji{x,1}\rangle =
\langle \bot, \bot,\ji{x,1}\rangle$.
\begin{lemma}
  The relations
  \begin{align*}
    \G{p} & =
    \bigvee_{1 < j \leq d} \ji{\myvec{j,p_{j},0}} \,, &
    \F{p} & = \bigvee_{1 \leq i < d} \ji{\myvec{i,p_{i},1}}  \,,
  \end{align*}
  hold in $\LId$.
\end{lemma}
\begin{proof}
  Recalling that $\F{p} = \Fam{f_{i,j} \mid (i,j) \in \cd}$ with
  $f_{i,j} = \ji{p_{i},1}$, we can compute within $\PrLI$ as follows:
  \begin{align*}
    \F{p} & = \bigvee_{(i,j) \in \cd} E^{i,j}_{p_{i},1}
    = \bigvee_{1 \leq i_{0} < d} \;\bigvee_{i_{0} < j \leq d}
    E^{i_{0},j}_{p_{i},1} \\
    & = \bigvee_{1 \leq i_{0} < d} \;\;(\;\bigvee_{i_{0} < j \leq d}
    E^{i_{0},j}_{p_{i_{0}},1}  \vee \bigvee_{(i,j) \in \cd, i \neq
      i_{0}} E^{i,j}_{1,1}\;)
    = \bigvee_{1 \leq i_{0} < d} \ji{\myvec{i_{0},p_{i_{0}},1}}\,.
  \end{align*}
  Again, Remark~\ref{rem:joins} ensures that the relation so derived
  holds in $\LId$ as well.
  The proof that $\G{p} = \bigvee_{1 < j \leq d}\ji{\myvec{j,p_{j},0}}$ is
  analogous.
\end{proof}

\begin{lemma}
  For each $i \in \setd$, $x \in \I$ and $y \in \set{0,1}$,
  $\ji{\myvec{i,x,y}}$ is a join of elements in $\LfId$.
\end{lemma}
\begin{proof}
  Let us consider the case where $y = 1$ (the proof when $y = 0$ is
  similar).

  If $x = 1$, then $\ji{\myvec{i,x,y}}$ already belongs to $\LfId$.
  If $x \neq 1$, 
  then all the coordinates different from $i$ are rational. If $x$ is
  not rational, then we can choose a descending sequence $r_{n}$ of
  rational numbers such that $\bigwedge_{n \geq 0} r_{n} = x$. Then,
  using the characterization of the order given in
  Corollary~\ref{cor:orderonjp},
  we see that the relation
  $\bigvee_{n \geq 0} \ji{\myvec{i,r_{n},1}} = \ji{\myvec{i,x,1}}$
  holds in $\PrLI$.  A fortiori, the same relation holds in $\LId$.
\end{proof}

We can summarize our observations with the following statement:
\begin{proposition}
  Every \mirr element of $\LId$ is a join of elements from $\LRId$.
\end{proposition}

Let in the following $\Sigma_0(\LfId) = \Pi_0(\LfId) = \LfId$ be the set
of tuples that have discrete rational functions as components.  Let
$\Sigma_{n+1}(\LfId)$ be the closure under joins of $\Pi_n(\LfId)$; let
$\Pi_{n+1}(\LfId)$ be the closure under meets of $\Sigma_n(\LfId)$.
\begin{theorem}
  Every element of $\LId$ belongs both to $\Sigma_{2}(\LfId)$ and
  $\Pi_{2}(\LfId)$.
\end{theorem}
\begin{proof}
  By Corollary~\ref{cor:joinofeps} and the fact that $\LId$ is
  autodual, every element of $\LId$ is a meet of \mirr elements. We
  have seen above that each \mirr element is a join of elements from
  $\LfId$ so it belongs to $\Sigma_{1}(\LfId)$. It follows that every
  element of $\LId$ is an element of $\Pi_{2}(\LfId)$. Since $\LId$
  and $\LfId$ are autodual, this also proves that every element of
  $\LId$ is an element of $\Sigma_{2}(\LfId)$.
\end{proof}

Using the terminology of \cite{GehrkeHarding2001}, the previous
theorem states that $\LRId$ is dense in $\LId$. Yet, $\LId$ is not a
canonical extension of $\LRId$. A canonical extension of a lattice is
a complete spatial lattice, meaning that every element is the infinite
join of the \cjirr elements below it, see
\cite[Lemma~3.4.]{GehrkeHarding2001}. As argued in
Section~\ref{subsec:lackspatials}, there are no \cjirr
elements in $\LId$, in particular the lattices $\LId$ are not spatial.

\bibliographystyle{abbrv}
\bibliography{biblio}

\end{document}